\documentclass[12pt]{amsart}
\usepackage{amssymb,amsmath,amsthm,mathtools,amscd,mathrsfs,overpic}
\usepackage{holtpolt}
\usepackage{enumitem}
\usepackage[margin=1in]{geometry}
\usepackage{mleftright}
\usepackage{soul}
\usepackage{pdfpages}
\mleftright

\usepackage{nccmath}

\def\B{\mathcal{B}}
\def\Z{\mathbb{Z}}
\def\Q{\mathbb{Q}}
\def\R{\mathbb{R}}

\def\B{\mathcal{B}}

\def\Sc{\mathcal S}

\def\a{\alpha}
\def\b{\beta}
\def\t{\tau}

\def\g{\gamma}
\def\d{\delta}
\def\l{\lambda}
\def\G{\Gamma}
\def\D{\Delta}

\def\half{\tfrac{1}{2}}

\DeclareMathOperator{\sgn}{sgn}
\def\SL{{\rm SL}}
\def\GL{{\rm GL}}
\newcommand*{\defeq}{\stackrel{\text{def}}{=}}

\newcommand{\pMatrix}[4]{\left(\begin{matrix}#1 & #2 \\ #3 & #4\end{matrix}\right)}

\renewcommand{\pmatrix}[4]{\left(\begin{smallmatrix}#1 & #2 \\ #3 & #4\end{smallmatrix}\right)}

\def \ep {\varepsilon}

\newtheorem*{theorem1*}{Dirichlet Approximation Theorem}
\newtheorem*{theorem2*}{Minkowski Approximation Theorem}
\newtheorem*{theorem3*}{Minkowski's First Convex Body Theorem}
\newtheorem*{theorem*}{Theorem}
\newtheorem{theorem}{Theorem}
\newtheorem{lemma}{Lemma}

\newtheorem*{corollary*}{Corollary}

\newtheorem{definition}{Definition}
\theoremstyle{remark}

\numberwithin{equation}{section}
\numberwithin{lemma}{section}
\title{On a theorem of Davenport and Schmidt
}
\date{\today}

\author{Nickolas Andersen}
\address{UCLA Mathematics Department,
Box 951555, Los Angeles, CA 90095-1555} \email{nandersen@math.ucla.edu}

\author{William Duke}
\address{UCLA Mathematics Department,
Box 951555, Los Angeles, CA 90095-1555} \email{wdduke@ucla.edu}
\thanks{Supported by NSF grant DMS 1701638.}

\begin{document}

\begin{abstract}
This work is motivated by a paper of Davenport and Schmidt,
which treats  the question of  when Dirichlet's theorems on the rational approximation of one or of two irrationals can be improved and if so, by how much. 
We  consider a generalization of this question in the simplest case of a single irrational but 
 in the context of the geometry of numbers in $\R^2$, with the sup-norm replaced by a more general one.
  Results include sharp bounds for how much improvement is possible under various conditions.
The proofs use  semi-regular continued fractions that are characterized by a certain best approximation property determined by the norm.
\end{abstract}

\maketitle

\section{Introduction }\label{intro}

In 1842 Dirichlet \cite{Dir}  applied the  pigeonhole  principle 
  to give good approximations of real numbers by rationals.
One form of his theorem in one dimension is the following.
\begin{theorem1*} 
For $\a\in \R$ and any  $Q\in \Z^+$ there are  $p,q\in \Z$ such that $1\leq q\leq Q$ and
$|p-q\a|< \tfrac{1}{Q}.$ 
\end{theorem1*}
\noindent

 Davenport and Schmidt \cite{DS}  considered those $\a$ for which an improvement of this result is possible, at least when we only require that $Q$ be sufficiently large.   More precisely, 
let $\d(\a)$ be the largest number with the property that 
if $c>\d(\a)$ then for  
 {\it every} sufficiently large $Q$ (depending only on $\a$), there are integers
$p,q\in \Z$ with $1\leq q\leq
Q$ and
$
Q|p-\a q|< c,
$
while if $c<\d(\a)$ there are arbitrarily large $Q$ for which no such $p,q$ exist.
If $\d(\a)<1$ then we say that an improvement on Dirichlet's theorem is possible for this $\a$.
Clearly $\d(\a)=0$ for rational $\a$ so we only consider irrational $\a$.

An easy direct argument proves the fact, perhaps surprising at first,  that any irrational $\a$ for which $\d(\a)<1$ must be {\it badly approximable}.
For $\a$ to be   badly approximable means that   for some $c>0$ we have $|\a-\frac{p}{q}|>\frac{c}{q^2}$  for all relatively prime integers $p,q$ with $q>0$.
Davenport and Schmidt gave another proof of this that also shows that, conversely, an improvement on Dirichlet's theorem is possible for  every badly approximable number.  They deduced this from a formula for $\d(\a)$ given in terms of the regular continued fraction expansion of $\a.$
Recall that an irrational $\a$ has a unique infinite 
regular continued fraction expansion 
\begin{equation}\label{scf2}
\a=b_0 +\frac{1}{b_1+}\;\frac{1}{b_2+}\cdots \defeq b_0+\cfrac{1}{b_1 +\cfrac{1}{b_2 +\cfrac{1}{\ddots }}},
\end{equation}
where the  partial quotients $b_n$ satisfy $b_0=\lfloor\a\rfloor$ and $ b_k\in \Z^+$ for $ k\geq 1$.
Also define $u_0=\a-a_0$, $v_0=0$ while for $n\geq 1$ let
\begin{equation}\label{unvnp}
u_n=\frac{1}{b_{n+1}+}\;\frac{1}{b_{n+2}+}\cdots\;\;\mathrm{and}\;\;v_n=\frac{1}{b_{n}+}\;\frac{1}{b_{n-1}+}\;\frac{1}{b_{n-2}+}\cdots \frac{1}{b_1}.
\end{equation}

\begin{theorem*}(Davenport-Schmidt \cite{DS})
For any irrational $\a\in \R$ we have that
\begin{equation}\label{Del}
\d(\a)=\limsup_{n\rightarrow \infty}\big(1+ u_n v_n\big)^{-1}.
\end{equation}
\end{theorem*}
An immediate consequence of (\ref{Del}) is that the irrational $\a\in\R$  for which Dirichlet's theorem can be improved  are precisely those whose 
continued fraction have bounded partial quotients. This condition is well-known to be equivalent to $\a$ being badly approximable
 \cite[ p. 22]{Sch1}.
Real quadratic irrationalities are precisely those whose regular continued fraction expansions are eventually periodic, 
so they are badly approximable.
On the other hand, they are the only known examples that are algebraic. 
A continued fraction discovered by  Euler \cite{Eul} provides an explicit example of an irrational (in fact transcendental) number that is not badly approximable,
namely
\begin{equation}\label{well}
\frac{e-1}{e+1}=\frac{1}{2+}\;\frac{1}{6+}\;\frac{1}{10+}\;\frac{1}{14+}\cdots.
\end{equation}

By a  well-known result of  Khintchine \cite[Thm 29]{Khi}  badly approximable numbers, although uncountable,  are rare in the sense of measure theory.
Thus we have the following.
\begin{corollary*}\label{c2}
The set of  real irrationals for which Dirichlet's theorem can be improved is uncountable and has Lebesgue measure zero.
\end{corollary*}

Another   consequence of the formula (\ref{Del}) is a bound for how much the Dirichlet theorem can be improved
when it can be improved at all.\footnote{ For further results about the set of values of $\d(\a)$ see \cite{Iv} and the references therein. See also  our \S \ref{relate}.}
\begin{corollary*}\label{c3}
The smallest value of $\d(\a)$ is given by
\begin{equation}\label{lower}
\d(\a)=\tfrac{1}{10} (\sqrt{5}+5)=0.723607\dots,
\end{equation}
when $\a=\frac{1}{2}(1+\sqrt{5})$.
\end{corollary*}

 \section{Improving the Minkowski approximation theorem }

Davenport and Schmidt used their theorem as a starting point to obtain results that pertain to the Dirichlet theorems about approximating two numbers simultaneously and later  more generally  \cite{DS2} (see also \cite{Sch0}).
In this paper we will consider a different kind of generalization of Dirichlet's results, one that was conceived of by Hermite and Minkowski. 

Let $F:\R^2\rightarrow \R$ be a fixed norm on $\R^2$ and $\B$ its unit ball.
Define the stretched norm $F_t$ for $t>0$ by
 \begin{equation}\label{newn}
F_t(x,y)=F(t^{-1}x,t y). 
\end{equation}

The following generalization of Dirichlet's theorem follows from the work of Minkowski. 
Although it was  not stated directly by him, for the purposes of this paper we will refer to it as the {\it Minkowski approximation theorem} (in two dimensions).
\begin{theorem2*}
For a fixed norm $F$ on $\R^2$ let $\Delta=\Delta_F$ be the minimal area of a parallelogram with one vertex at the origin and the other three on the boundary of $\B$.
Fix $\a\in \R$. Then for any real $t\geq1$ there exist integers $p,q$ with $q>0$  such that
\begin{equation}\label{mi1}
\Delta\,F^2_t(q,p-\a q)\leq 1.
\end{equation}
\end{theorem2*}
Note that for this result we are not restricting $t$ to be an integer.
It is not hard to see that for the sup-norm the Minkowski approximation theorem  implies Dirichlet's theorem.
In this case $\D=1.$

The idea of generalizing Dirichlet's theorem to other norms goes back at least to Hermite \cite{Her}. He applied
({\ref{mi1}) for the Euclidean norm, for which $\D=\frac{\sqrt{3}}{2}$,  together with the inequality between arithmetic and geometric means. 
The resulting inequality implies that  for any irrational $\a$ there are infinitely many integers $p,q$ with $q>0$ such that
\begin{equation}\label{hrm}
q|p-\a q|< \tfrac{1}{\sqrt{3}},
\end{equation}
improving upon the corresponding upper bound $1$ given by Dirichlet's theorem.
Later  Minkowski \cite{Mink1,Mink2} showed that (\ref{mi1})  with the 1-norm given by $F(x,y)= |x|+|y|$  and for which $\D=\frac{1}{2}$,
implies (\ref{hrm}) with $\tfrac{1}{\sqrt{3}}$ replaced by  $\frac{1}{2}$.

Given these results of Hermite and Minkowski, it  is  natural to study the generalization for any norm of the quantity $\d(\a)$ from the Davenport-Schmidt theorem.
We want this generalization  to measure to what extent the Minkowski approximation theorem (\ref{mi1}) can be improved
for a particular $\a$.
Hence  for a fixed norm   $F$, 
let  $\d_F(\a)$ be the largest number with the property that 
if $c>\d(\a)$ then {\it for every sufficiently large} $t$
 there are 
$p,q\in \Z$ with $q>0$ such that 
\[
\Delta\,F_t^2(q,p-\a q)< c,
\]
while for $c<\d_F(\a)$ 
 there are arbitrarily large $t$ for which  no such $p,q$ exist.
For a given norm we say that the  Minkowski approximation theorem can be improved for irrational $\a\in \R$ if $\d_F(\a)<1.$
A straightforward argument shows that when $F$ is the sup-norm, $\d_F=\d$ for $\d$  in the Davenport--Schmidt theorem.

We have only been able to obtain satisfactory results about $\d_F$ if we  make the 
 assumption that for all $(x,y)\in \R^2$ the norm $F$ satisfies
\begin{equation}\label{SS}
F(x,y)=F(|x|,|y|).
\end{equation}
 At least  for the study of $\d_F$,  we may assume without any further loss of generality that the norm $F$ also satisfies
\begin{equation}\label{normal}
F(0,\pm1)=F(\pm1,0)=1.
\end{equation}

\begin{definition}
Say that a  norm $F$ is  {\it strongly symmetric} if it satisfies  (\ref{SS}) and (\ref{normal}).
\end{definition}

The most important  strongly symmetric norms are  the $p$-norms.  For $(x,y)\in \R^2$ and a fixed $1\leq p< \infty$ the $p$-norm is defined by
\[
F^{\langle p\rangle}(x,y)=(|x|^p+|y|^p)^{\frac 1p},
\]
while $F^{\langle \infty\rangle}(x,y)=\sup\{|x|,|y|\}.$
 Denote the corresponding $\B$ by $\B^p$, $\D$ by $\D_p$ and  $\d$ by $\d_p.$
Other interesting examples are the two unique strongly symmetric norms  whose unit balls are regular octagons:
$\B^{\mathrm{oct_1}}$ and $\B^{\mathrm{oct_2}}$  (see Figure~\ref{fig:oct}).

\begin{figure}[h]\label{normfig}
    \centering
    \begin{overpic}[width=0.2\textwidth]{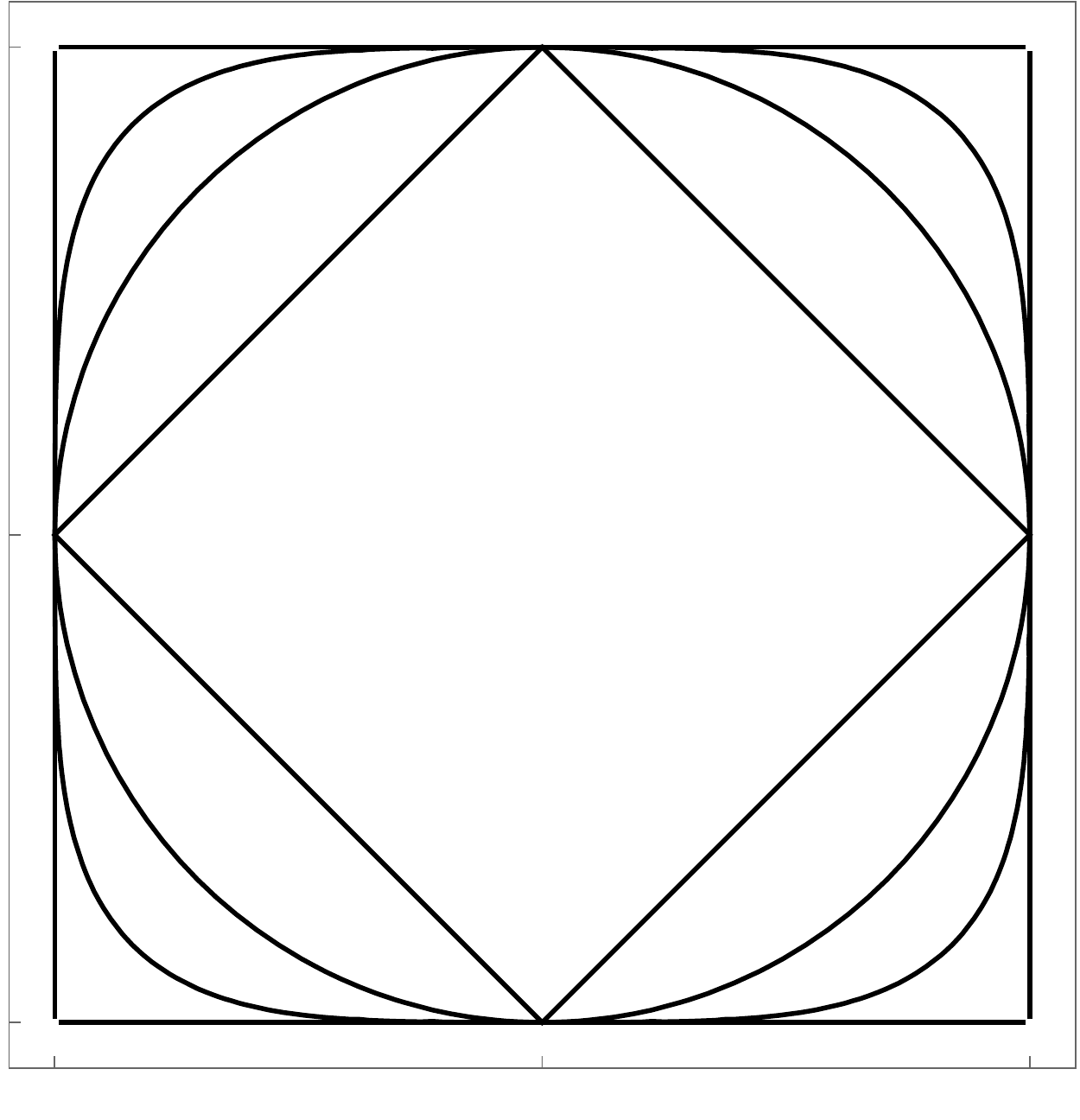}
    \tiny
    \put(-1,-3){$-1$}
   	\put(46.4,-3){$0$}
   	\put(90,-3){$1$}
   	\put(-12, 6){$-1$}
   	\put(-6, 49.5){$0$}
   	\put(-6,93){$1$}
   \end{overpic}
            \qquad
    \begin{overpic}[width=0.2\textwidth]{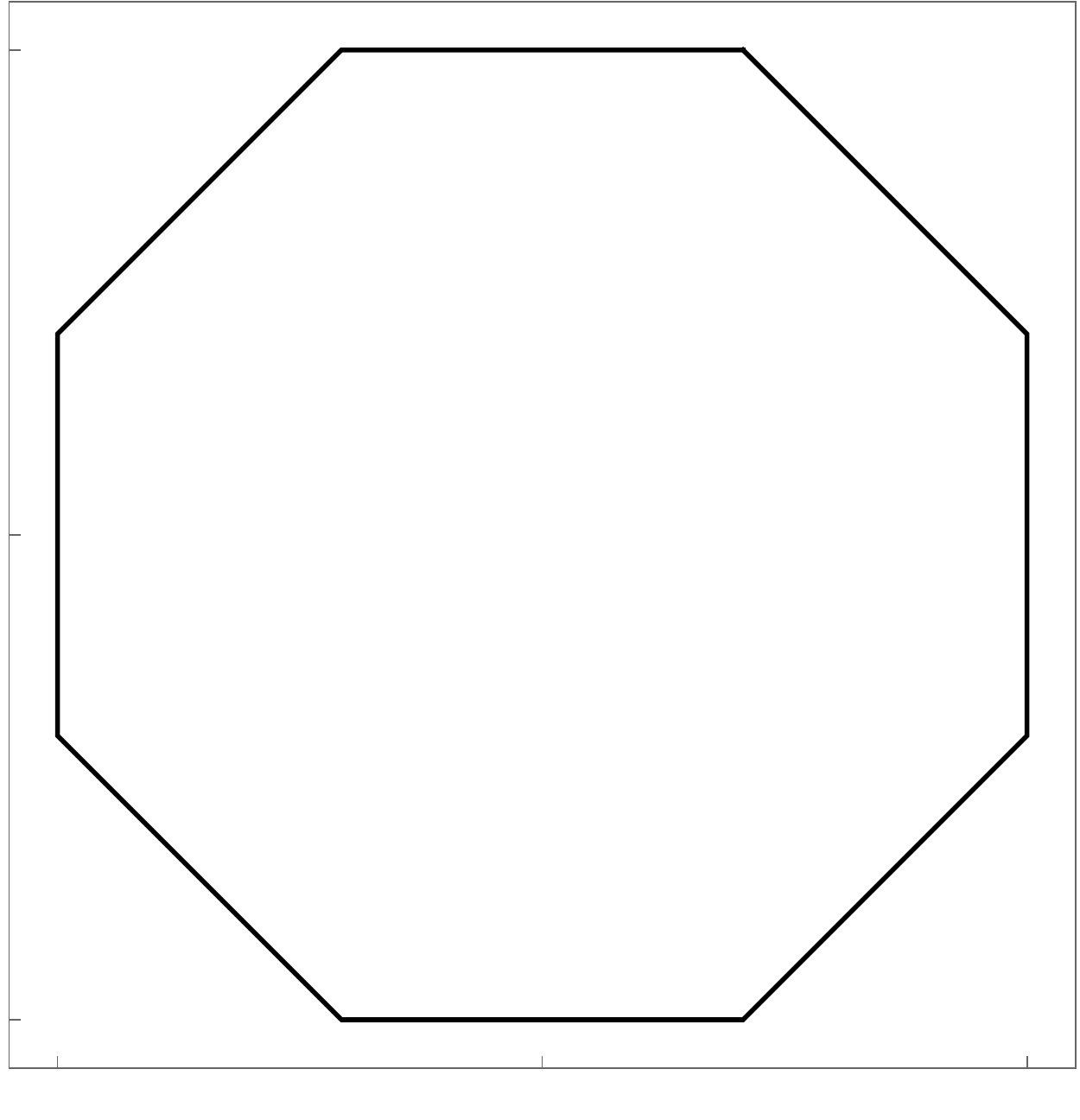}
    \tiny
    \put(-1,-3){$-1$}
   	\put(46.4,-3){$0$}
   	\put(90,-3){$1$}
   	\put(-12, 6){$-1$}
   	\put(-6, 49.5){$0$}
   	\put(-6,93){$1$}
   \end{overpic}
             \qquad
    \begin{overpic}[width=0.2\textwidth]{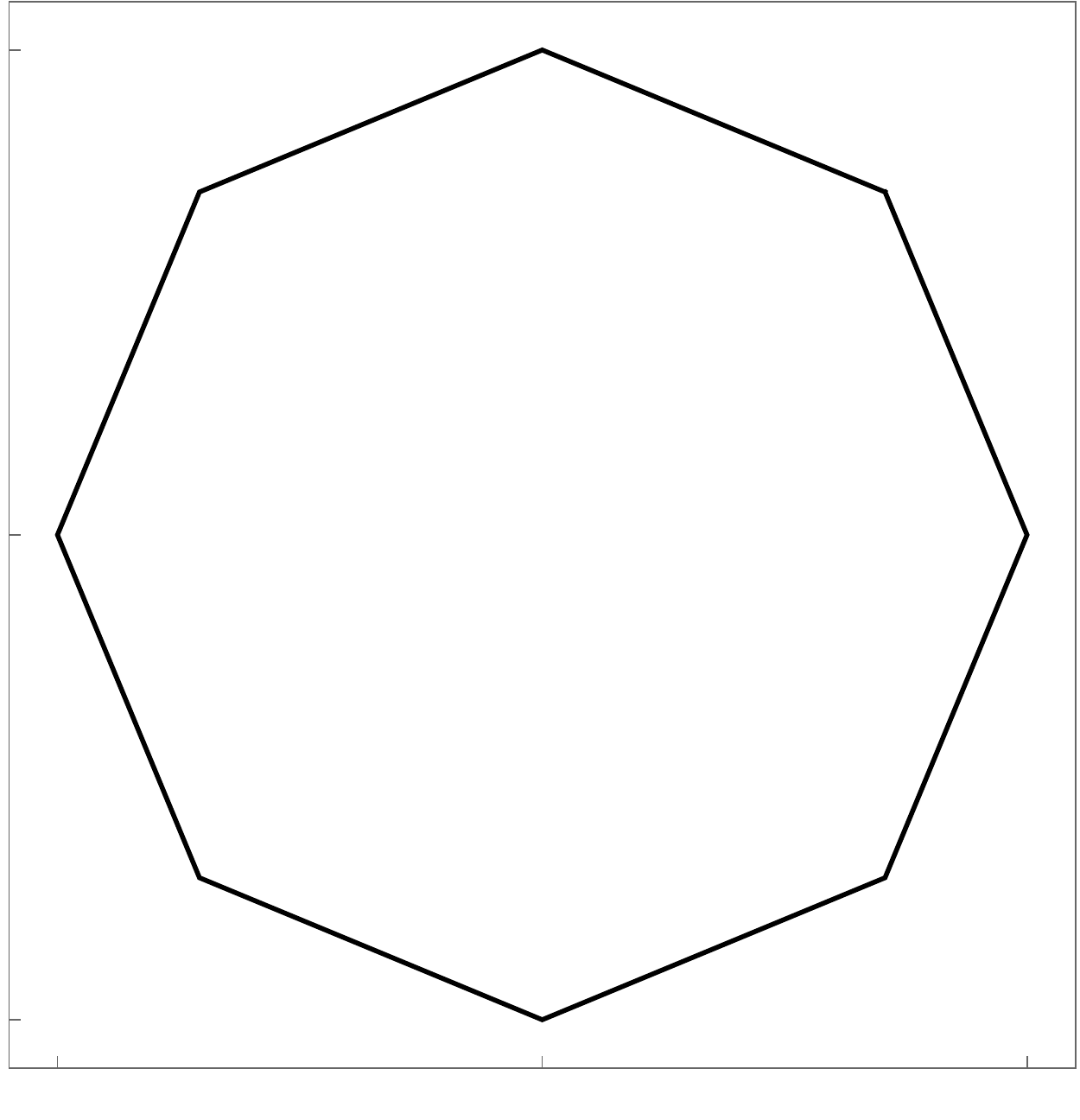}
    \tiny
    \put(-1,-3){$-1$}
   	\put(46.4,-3){$0$}
   	\put(90,-3){$1$}
   	\put(-12, 6){$-1$}
   	\put(-6, 49.5){$0$}
   	\put(-6,93){$1$}
   \end{overpic}
     \caption{ $\B^p$ for $p=1,2,4,\infty$ and $\B^{\mathrm{oct_1}}$ and $\B^{\mathrm{oct_2}}$.}
     \label{fig:oct}
 \end{figure}
Our first result generalizes the first corollary of the theorem of Davenport and Schmidt. It  shows that for a strongly symmetric norm the set of irrationals for which the
 Minkowski approximation theorem can be improved, while uncountable,  is small in the sense of measure theory.
\begin{theorem} \label{t4}
Fix a strongly symmetric norm $F$.
Then the set of all real irrationals for which Minkowski's approximation theorem can be improved is uncountable and has Lebesgue measure zero.
\end{theorem}

 Next we have  a uniform lower bound for $\d_F(\a)$ for any strongly symmetric norm and any irrational $\a$.
\begin{theorem} \label{genl}
For any strongly symmetric norm $F$
and any irrational $\a\in \R$ we have that
\begin{equation}\label{half2}
\d_F(\a) \geq \tfrac{1}{2}.
\end{equation}
\end{theorem}
Equality in (\ref{half2}) can hold for the 1-norm. This follows from the next result since $\D_{1}=\frac 12.$
For simplicity say that an irrational  $\a\in \R$ is {\it well approximable} if it is  not badly approximable.
\begin{theorem} \label{genl2}
For any strongly symmetric norm $F$ the smallest value of $\d_F(\a)$ for a well approximable
 $\a$  is $\Delta.$
\end{theorem}
We will see in the proof of Theorem \ref{genl2}  that $\d_p(\a)=\Delta$  for any  $\a$ whose regular continued fraction has partial quotients that are eventually  strictly increasing, for example  $\a=\tfrac{e-1}{e+1}$ from (\ref{well}).
For the $p$-norm we can go further and identify the smallest value of $\d_p(\a)$ for any irrational $\a$.  \begin{theorem} \label{new2}
For the $p$-norm the smallest value of $\d_p(\a)$ for an irrational $\a$
is $\D_p$ when $1\leq p \leq 2$ and is 
\begin{equation}\label{Dp}
\tfrac{\D_p}{10} \left(\sqrt{5}+5\right) \left(\left(\tfrac{1}{2} (\sqrt{5}-1)\right)^p+1\right)^{2/p},
\end{equation}
when $2<p\leq \infty.$ The value in (\ref{Dp}) is attained when $\a=\frac{-1+\sqrt{5}}{2}.$
 \end{theorem}
The value of $\Delta_p$  is given below in (\ref{pn1}).
See Figure \ref{fig:dp} for graphs of $\D_p$ and the minimum value of $\d_p.$
 It  is not the case that the Minkowski approximation theorem can always be improved
for each badly approximable irrational, not even each real quadratic irrational.  
For example, we show  at the end of \S\ref{sec:s-exp} that
 \begin{equation}\label{d2}
 \d_2\big(\tfrac{1}{2}(-1+\sqrt{3})\big)=1. 
 \end{equation}

Finding the norm or norms with the  largest minimum  value of $\d_F(\a)$ among all strongly symmetric norms  seems  an interesting problem.
The 2-norm has the largest minimum  value $\frac{\sqrt{3}}{2}=0.866025\dots $  of $\d_p$ among all $p$-norms.   It can be shown that the minimum value of $\d_{F}(\a)$ for both of the octagonal norms  is $\frac{1}{8} \left(3 \sqrt{2}+2\right)=0.78033\dots$ (see the end of \S \ref{small}).
Among all of the examples we have considered,  the 2-norm provides the largest minimum (see Figure~\ref{fig:dp}). 
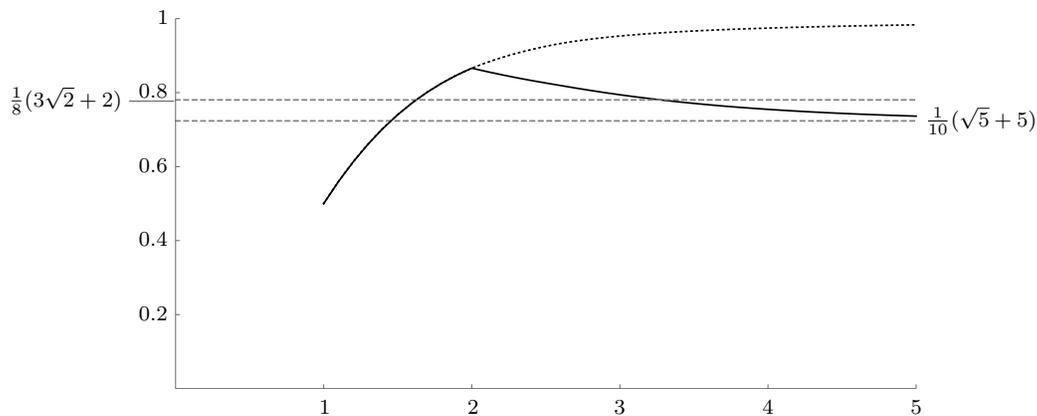
\begin{figure}[h]
    \centering
    \begin{overpic}[width=0.6\textwidth]{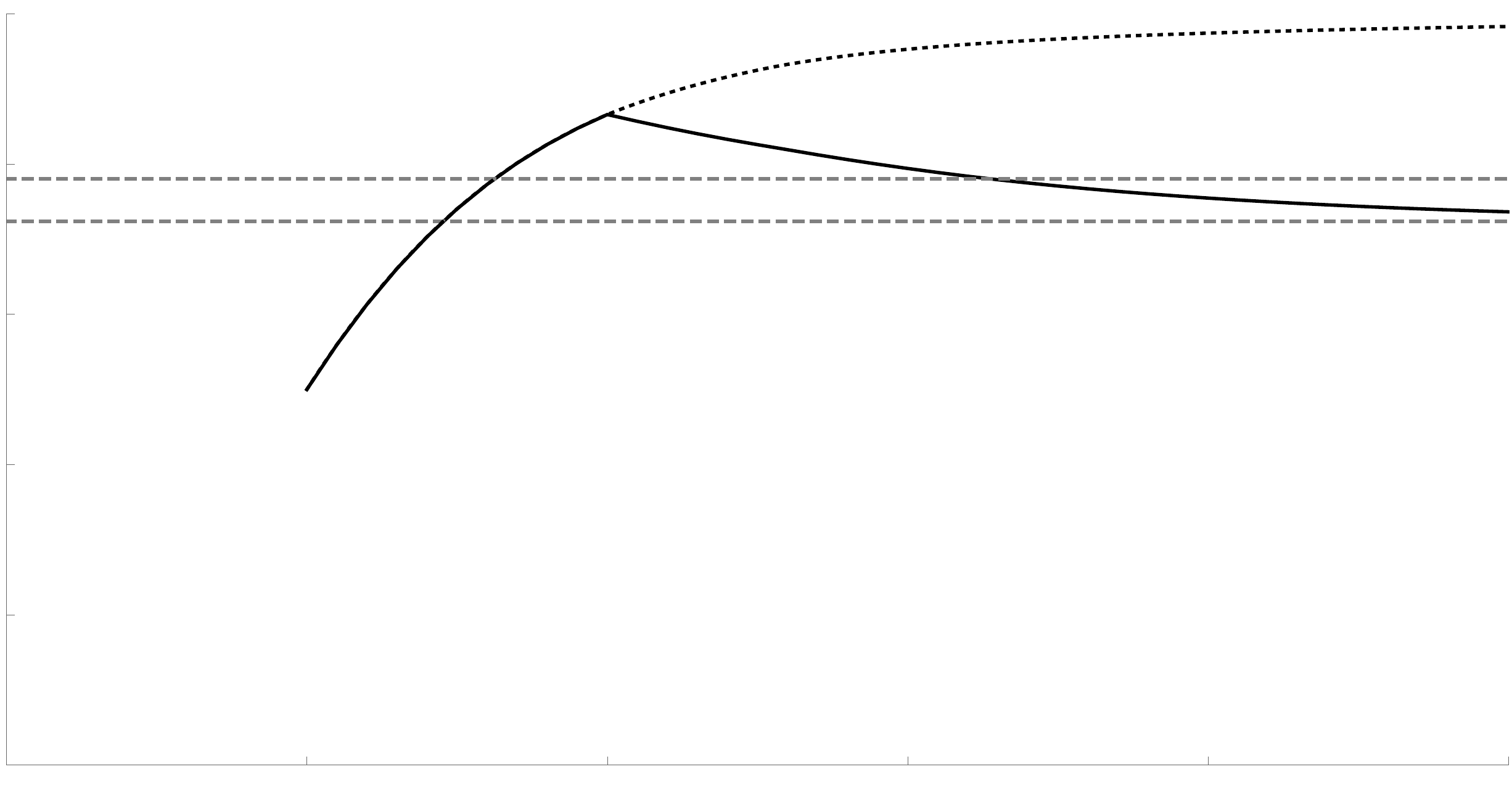}
    \tiny
    \put(19.7,-1){$1$}
    \put(39.6,-1){$2$}
    \put(59.3,-1){$3$}
    \put(79.2,-1){$4$}
    \put(99.1,-1){$5$}
    \put(-4.5,11.5){$0.2$}
    \put(-4.5,21.5){$0.4$}
    \put(-4.5,31.4){$0.6$}
    \put(-4.5,41.4){$0.8$}
    \put(-2,51.3){$1$}
    \put(101,37.6){$\tfrac 1{10}(\sqrt 5+5)$}
    \put(-22,40.3){$\tfrac 1{8}(3\sqrt 2+2)$ ------}
   \end{overpic}    
     \caption{The minimum of $\d_p$ for $p\geq 1$. The dotted line is $\Delta_p$, the minimum of $\d_p(\alpha)$ for well approximable $\alpha$.}
    \label{fig:dp}
\end{figure}

\subsubsection*{Remarks:}

 Results like ours involving continuously varying norms  belong  to the ``parametric geometry of numbers," an area that has recently seen a revival of activity
 stimulated by Schmidt and Summerer \cite{SS, SS1}.
See also 
\cite{Ro} and its references.

The sequence $(u_n,v_n)$ from (\ref{unvnp}) represents a trajectory of the dynamical system on
 $\Omega_0=[0,1)\times [0,1]$ determined by the extended continued fraction map 
 $T:\Omega_0\rightarrow\Omega_0$ given by  $T(0,v) = (0,0)$ and for $u>0$ by
 \begin{equation} \label{eqT}
T(u,v) = \left( \frac 1u - \left\lfloor\frac 1u\right\rfloor, \frac{1}{v+\lfloor\tfrac 1u\rfloor} \right).
\end{equation}
It has an invariant measure $\omega$ 
with density function 
\begin{equation}\label{ome}
 \frac{1}{\log 2}  \frac{1}{(1+uv)^2}.
	\end{equation}
The ergodicity of this system, which is the natural extension of the usual continued fraction dynamical system,   can be used to give a different proof that  Dirichlet's theorem cannot be improved for almost all real irrationals.  An argument of  \cite{Jag}, given as Lemma 5.3.11 of  \cite{DK},   allows one to conclude an almost all result for the special trajectories (\ref{unvnp}).
Our proof of Theorem \ref{t4} proceeds along similar lines except that the trajectories of our dynamical system  are determined by certain  semi-regular continued fractions, which admit $\pm 1$ as partial numerators. The continued fractions we need are examples of $\Sc$-expansions,
which have a well developed metrical theory again based  on the ergodic theorem. 
For some remarks on the connection between these dynamical systems and the geodesic flow on $\SL(2,\Z)\backslash \SL(2,\R)$ see \S \ref{relate}.

The second corollary of the Theorem of Davenport and Schmidt and our generalization, Theorem \ref{new2},
require for their proofs information about {\it all},  rather than almost all trajectories. 
Other aspects of the continued fractions we use are needed, including a  best approximation property given in terms of the norm,
in order to be able to analyze in detail each individual trajectory.

%

 \section{The continued fraction associated to a norm }

We want to give  a generalization of  the  formula (\ref{Del})  of   Davenport and Schmidt and
for that we  require, as previously mentioned,  certain infinite {\it semi-regular}  continued fraction expansions.   Such a continued fraction has the form
\begin{equation}\label{srcf1}
a_0+\frac{\ep_1}{a_1+}\,\frac{\ep_2}{a_2+}\;\frac{\ep_3}{a_3+}\cdots,\;\;\; \ep_m=\pm 1,  \,a_m\in \Z
\end{equation}
where $a_m>0$ and
$a_m+\ep_{m+1} \geq 1$
for all $ m \geq 1$ and  $a_m+\ep_{m+1} \geq 2 $ for infinitely many $m$.
For any $m \geq 0$ the $m^{th}$ convergent of this continued fraction 
\[
\frac{p_m}{q_m}=a_0+\frac{\ep_1}{a_1+}\,\frac{\ep_2}{a_2+}\cdots\frac{\ep_m}{a_m}
\]
uniquely defines relatively prime integers $p_m,q_m$ with $q_m>0$, where  $p_0=a_0$ and $q_0=1.$
 Tietze (\cite{Ti}, see also \cite[p.~135]{Per}) showed that there is an irrational $\a$ to which such a continued fraction converges, meaning that
$\a=\lim_{m\rightarrow \infty} \frac{p_m}{q_m}.$
%

The continued fraction we need is  characterized by a best approximation property stated in terms of the given strongly symmetric  norm.
\begin{definition}\label{def1}
Say that a rational number $\frac pq$ where $q>0$ is a best approximation to $\alpha$ with respect to the norm $F$ if there is a $t>1$ depending only on  $\frac pq$ such that
\begin{equation*}
	F_t(q,p-\alpha q) < F_t(s,r-\alpha s)
\end{equation*}
for all rational $\frac rs\neq \frac pq$.
\end{definition}

In the case of the sup-norm Definition \ref{def1}  is equivalent to the usual one that states that a  rational number $\frac pq$ with $q>0$  is a  best approximation to an irrational $\alpha$ if for all rational numbers $\frac rs\neq \frac pq$ with $0<s\leq q$ we have
\begin{equation*}
	|p-\alpha q| < |r-\alpha s|
\end{equation*}
(see Lemma \ref{bap} below).

\begin{theorem}\label{tcf}
Fix a strongly symmetric norm $F$. Every irrational $\a\in \R$ has  a unique
semi-regular continued fraction expansion  whose convergents are precisely the best approximations to $\alpha$ with respect to $F$.
  \end{theorem}
We will refer to  this continued fraction as the {\it $F$-continued fraction of} $\a$  and,  for the $p$-norm, as the {\it $p$-continued fraction of $\a.$}
For the sup-norm the $\infty$-continued fraction is closely related to, but not always equal to, the regular continued fraction.
Suppose that
 \begin{equation}\label{rega}
\a=b_0+\frac{1}{b_1+}\,\frac{1}{b_2+}\;\frac{1}{b_3+}\cdots
\end{equation}
is the regular continued fraction of an irrational $\a$.
Recall that Lagrange showed (\cite{Lag}, see also \cite[\S 15]{Per})  that every best approximation in the usual sense is  a convergent of the regular continued fraction of $\alpha$ 
and that every convergent, except possibly $b_0$,  is a best approximation to $\alpha$.
In view of Theorem \ref{tcf}, (\ref{rega})  coincides with the $\infty$-continued fraction of $\a$  
if and only if $b_1>1$.  If $b_1=1$ the $\infty$-continued fraction of $\a$  is
\begin{equation}\label{eq:sing-intro}
\a=b_0+1+\frac{-1}{b_2+1+}\,\frac{1}{b_3+}\cdots.
\end{equation}
This is an example of a singularization, which has the effect of contracting the regular continued fraction by removing $b_1=1$ and the convergent $b_0$,
which is {\it not} a best approximation to $\a$ in this case.  This well-known exceptional case  does not occur for the $\infty$-continued fraction.
With this one possible exception, however,  the convergents of the $\infty$-continued fraction and those of the regular continued fraction coincide. 

For any $1\leq p< \infty$,  the  inequality between arithmetic and geometric means immediately gives  
that a necessary condition for a regular convergent $\frac{r_n}{s_n}$  of an irrational $\a$ to be 
a convergent of the $p$-continued fraction is that
\begin{equation}\label{convm}
s_n|r_n -\a s_n|\leq (4^{1/p}\D_p)^{-1}.
\end{equation}
For $p=1$, when the right hand side is $\frac 12$, Minkowski \cite{Mink1.5} showed that (\ref{convm}) is also sufficient.

Just as the formula (\ref{Del})  is given in terms of the sequence $u_n,v_n$ coming from the regular continued fraction,
our generalization will be given in terms of a sequence $\mu_m,\nu_m$   determined by our continued fraction
 $\a=a_0+\frac{\ep_1}{a_1+}\,\frac{\ep_2}{a_2+}\;\frac{\ep_3}{a_3+}\cdots$.  Namely, for a fixed norm we define
 $\mu_0=\a$ and $\nu_0=0$, while for $m \geq 1$ we let
\begin{equation}\label{unvn}
\mu_m=\frac{\ep_{m+1}}{a_{m+1}+}\;\frac{\ep_{m+2}}{a_{m+2}+}\cdots\;\;\mathrm{and}\;\;\nu_m=\frac{1}{a_{m}+}\;\frac{\ep_{m}}{a_{m-1}+}\;\frac{\ep_{m-1}}{a_{m-2}+}\cdots \frac{\ep_2}{a_1}.
\end{equation}
For a general strongly symmetric norm we will  express $\d_F(\a)$ in terms of these numbers $\mu_m,\nu_m$
in \S \ref{form} below.
For the $p$-norm the formula is completely explicit and we give it here.
For $p$ with $1\leq  p <\infty$ let
\begin{equation}\label{delta}
D_p(u,v)=\frac{1}{1+uv}\left(\frac{(1-|u|^pv^p)^2}{(1-|u|^p)(1-v^p)}\right)^{\frac{1}{p}},
\end{equation}
while when $p=\infty$ set $D_\infty (u,v)=\lim_{p\rightarrow \infty} D_p(u,v)=(1+uv)^{-1}.$

\begin{theorem}\label{t6}
Fix $1\leq p \leq \infty.$
For any irrational $\a$ whose $p$-continued fraction 
is (\ref{srcf1}) we have that
\[
\d_p(\a)=\limsup_{ m\rightarrow \infty}\D_p\,D_p\big(\mu_m,\nu_m \big),\]
where $\mu_m,\nu_m$ are given above in (\ref{unvn}). 
\end{theorem}

The $F$-continued fraction of an irrational $\a$ for any strongly symmetric norm is an example of an $\Sc$-expansion.
Their theory  has been developed by Kraaikamp  \cite{Kra2} and others (see also \cite{Bos}, \cite{DK} and \cite{IK}).
Recall the definition of $T$ and $\omega$  from (\ref{eqT}) and (\ref{ome}).
A  Borel set $\Sc \subset\Omega_0$ is called a singularization area if $\omega(\partial \Sc)=0$ 
and if
\begin{enumerate}[label=(\roman*)]
	\item $\Sc\subseteq [\tfrac 12,1)\times [0,1]$ and
	\item $T\Sc \cap \Sc \subseteq\{(\b,\b)\},$ where $\b=\frac 12(-1+\sqrt{5}).$
\end{enumerate}
The  $\Sc$-expansion of an irrational $\alpha$ is obtained from the regular continued fraction (\ref{scf2})
by changing
 \begin{equation*}
\a=	\cdots \frac{1}{b_{n}+} \, \frac{1}{1+} \, \frac{1}{b_{n+2}+} \cdots \;\;\;\;\text{into}\;\;\;\;\a=	\cdots \frac{1}{(b_{n}+1)+} \, \frac{-1}{(b_{n+2}+1)+} \cdots
\end{equation*}
for each $n$ such that $(u_n,v_n)\in \mathcal S$.
Note that $(u_n,v_n)\in \Sc$ implies that $b_{n+1} = 1$ by (i).
Also (ii) implies that this procedure is unambiguous.
The result is a unique semi-regular continued fraction for $\a$
whose convergents are precisely those regular convergents $\frac{r_n}{s_n}$ where $n\geq 0 $ is such that 
$(u_n,v_n)\notin \Sc$. For example, the $\infty$-continued fraction discussed above is the $\Sc$ expansion for $\Sc=[\frac 12,1)\times \{0\}$. As usual, we denote $\mathcal S$ by $\mathcal S_p$ in the case of the $p$-norm (see Figure \ref{ps}).
 \begin{theorem} \label{sexp} 
Fix a strongly symmetric norm $F$. There exists a singularization area $\Sc$ so that the $F$-continued fraction of any irrational  $\alpha$ 
is the  $\Sc$-expansion of $\a$.
\end{theorem}

\begin{figure}[h]\label{ps}
    \centering
    \begin{overpic}[width=0.35\textwidth]{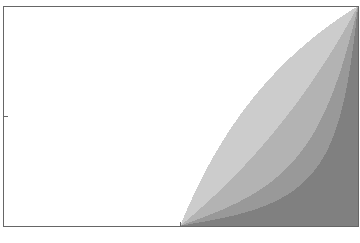}
    \tiny
    \put(48.5,-1){$\frac 12$}
    \put(98,-1){$1$}
    \put(0,-1){$0$}
    \put(-3,3){$0$}
    \put(-3,34){$\frac 12$}
    \put(-3,64){$1$}
    \end{overpic}
    \caption{The $\mathcal \Sc_p$ regions for $p=1,2,3,4$.}
    \label{fig:S}
\end{figure}
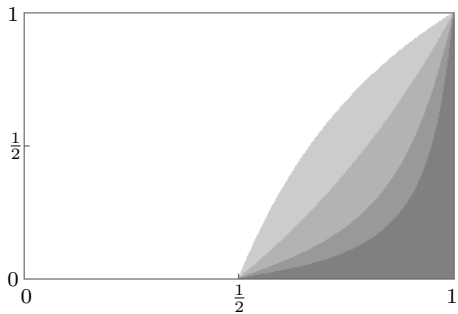

\noindent
{\it Remarks:} 
At the beginning of the paper
\cite{Mink1},  Minkowski states without proof several of  the main properties of the $p$-continued fraction for any $p$,
including the best approximation property.  
Our proof of Theorem \ref{tcf},  which  allows for $F$ to be any strongly symmetric norm, was strongly influenced by his ideas.
As previously mentioned, Minkowski \cite{Mink1.5}  also gave the  remarkable result  that  for the 1-continued fraction the necessary condition (\ref{convm}) is also sufficient.
Unsurprisingly, this also follows from our arguments.
In addition, the 2-continued fraction has actually been studied since the time of Hermite \cite{Her0}, especially by Humbert \cite{Hum1,Hum2}.
It is also closely connected to the improper modular billiards studied in \cite{AD}.
It was shown in  \cite{Kra1}
that the 1-continued fraction (known as Minkowski's diagonal continued fraction) is an $\Sc$-expansion.
For related work on the 1-continued fraction see \cite{Mosh}.
That the $p$-continued fraction for $p\neq 1, \infty$  is also an $\Sc$-expansion seems to be new.

\bigskip
In the next section we will review some basic facts from the geometry of numbers
in the case we need, namely  in two dimensions. Seven sections, each with the proof of one of our theorems, follow afterward.
The theorems will be proven in the following order:
\[\ref{genl}\rightarrow \ref{tcf} \rightarrow \ref{t6}\rightarrow \ref{sexp}\rightarrow \ref{genl2}\rightarrow \ref{new2} \rightarrow \ref{t4}.\]
Some concluding remarks are then given.
 Finally, an appendix  contains a number of technical lemmas and their proofs that we will refer to as needed
in the main body of the paper.

\section{Geometry of numbers}\label{GN}

As above let  $F$ be a fixed norm on $\R^2$.
This means that   for $P,P'\in \R^2$ we have
\begin{enumerate}[label=(\roman*)]
\item $F(P)\geq 0$ and $F(P)=0$ if and only if $P=(0,0)$
\item $F(tP)=|t|F(P)$ for $t \in \R$
\item $F(P+P')\leq F(P)+F(P')$.
\end{enumerate}
The unit ball of the norm is 
\[
\B=\{P\in \R^2;F(P)< 1\}.
\]
 This $\B$ is open, bounded, convex and symmetric around 0 and every such body arises as the unit ball of some norm
(see e.g. \cite{Sie}). 
Denote by $\mathrm{area}(\B)$  the  Lebesgue measure of $\B$ on $\R^2$.
It is convenient to define the stretched ball for $t>0$
\[
\B_t=\{(x,y)\in \R^2; F_t(x,y)< 1\}.
\]

Let $L\subset \R^2$ be a (full) lattice.
By the determinant of $L$, denoted $\det{L}$, we mean
$|\det{g}|$ for any $g\in \GL(2,\R)$ whose rows give a $\Z$-basis for $L$. 
The lattice $L$  is {\it admissible} for $\B$ if $\B$ contains no other points of $L$ than $(0,0)$.
The following result is fundamental \cite{Mink2}:
\begin{theorem3*}
If $L$ is  admissible for $\B$ then
 \[\mathrm{area}\,\B\leq4\det{L}.\]
\end{theorem3*}
The {\it critical determinant} of $\B$, denoted $\D(\B)$ or simply $\D$, is the infimum of all determinants of lattices admissible for $\B$. Building on work of Minkowski \cite{Mink1.7,Mink3}, 
Mahler \cite{Mah0} proved that lattices with determinant $\D$ actually exist, and these are called {\it critical lattices.}
Minkowski's first convex body theorem implies that
\begin{equation}\label{mi2}
\D\geq \tfrac{1}{4}(\mathrm{area}\,\B).
\end{equation}
This is sharp for the 1-norm and the sup-norm.

Apparently,  if we wish to evaluate $\d_F(\a)$ exactly we must also know $\D$ exactly.
Finding the critical determinant of a given $\B$
is the  main problem of the geometry of numbers in $\R^2.$
   Although the $n$-dimensional version of this problem is apparently intractable in general, here it is approachable.
  For a given critical lattice $L$ for $\B$ the boundary  of $\B$ must contain a $\Z$-basis $\{P,P'\}$  for $L$ as well as their sum $P+P'$.
 Furthermore, the lattice generated by  any pair of points $P,P'$ with   $P,P',P+P'$ on the boundary of $\B$ is admissible for $\B$ 
 (see \cite[Thm XI p.~160]{Cas}).
Therefore, as Minkowski already knew,  computing $\D$ amounts to solving  the (generally quite difficult) calculus problem of minimizing the area of a parallelogram with one vertex at the origin and the three others on the boundary of $\B$.  
This justifies our definition of $\D$ in the statement of the Minkowski approximation theorem.

Next we review what is known about the value of $\D_p$ for all $p$. Let
\[
\Delta_p^{(0)}=(1-2^{-p})^{\frac{1}{p}}\;\;\;\;\mathrm{and}\;\;\;\;\Delta_p^{(1)}=2^{-\frac{2}{p}}\,\tfrac{1+\t_p}{1-\t_p},
\]
where $0<\t_p<\frac{1}{2}$ satisfies $\t_p^p+1=2(1-\t_p)^p$.
A modification of a conjecture of
Minkowski \cite[p.~51--58]{Mink3} made by Davis \cite{Dav} states that
\begin{equation}\label{pn1}
\D_p=\min \{\Delta_p^{(0)},\Delta_p^{(1)}\}.
\end{equation}
Furthermore, there is a unique value $2.57<\rho<2.58$ so that $\D_p=\Delta_p^{(0)}$ when $2\leq p\leq \rho$, while otherwise $\D_p=\Delta_p^{(1)}$.
 Many  mathematicians obtained partial results,  among them Mordell  \cite{Mor2},   Davis \cite{Dav}, Cohn \cite{Co}, Watson \cite{Wa1,Wa2} and Malyshev \cite{Mal}.
Building on their work,   the proof  of the full conjecture was finally completed by  Glazunov,  Golovanov and Malyshev \cite{GGM}.
 \begin{figure}[h]
    \centering
    \begin{overpic}[width=0.4\textwidth]{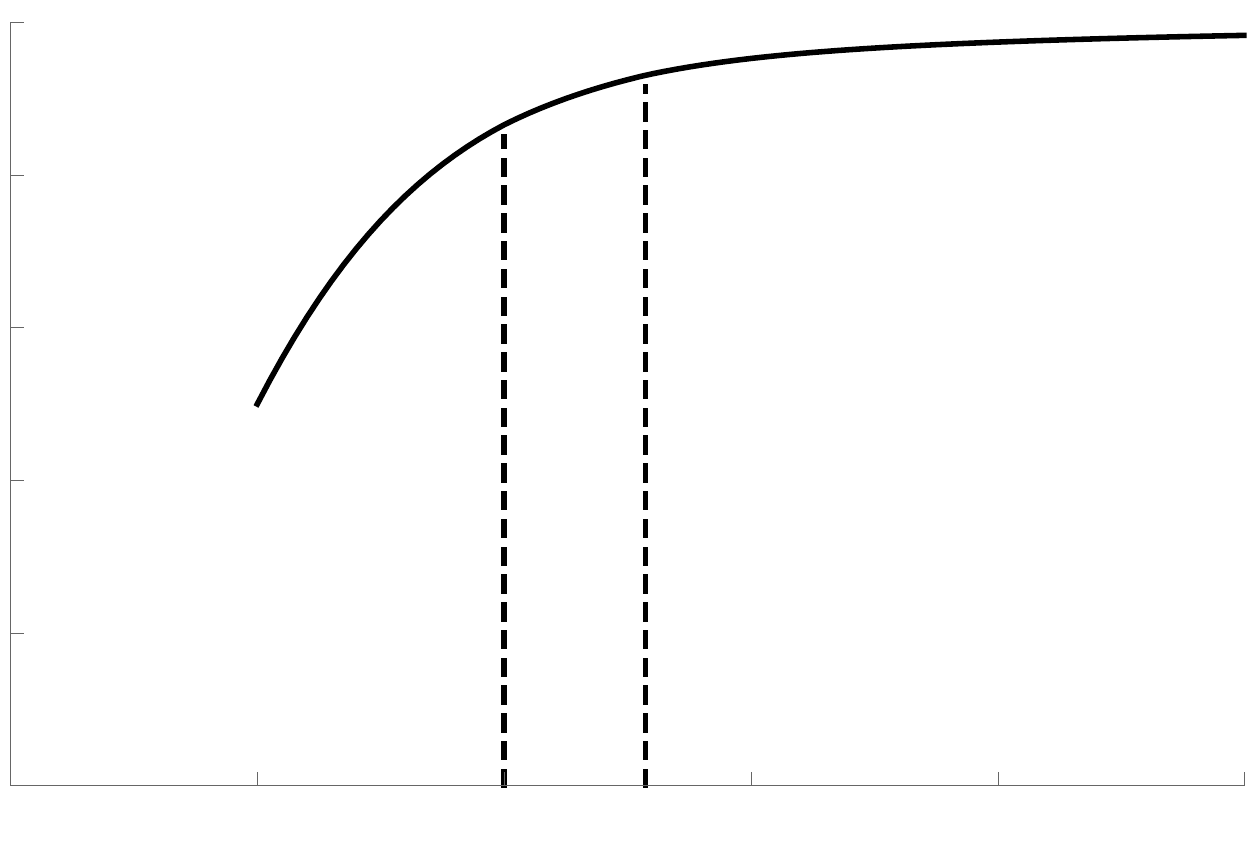}
    \tiny
    \put(0,0){$0$}
    \put(19,0){$1$}
    \put(39,0){$2$}
    \put(50,0.5){$\rho$}
    \put(59,0){$3$}
    \put(78,0){$4$}
    \put(98,0){$5$}
    \put(-7,15.5){$0.2$}
    \put(-7,27.5){$0.4$}
    \put(-7,40){$0.6$}
    \put(-7,52){$0.8$}
    \put(-3,64){$1$}
    \end{overpic}
    \caption{$\D_p$ for $1\leq p\leq 5$}
\end{figure}
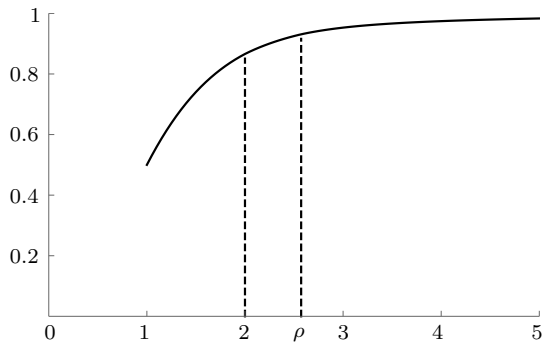

In the case of the $p$-norm parallelograms that minimize the area may be given explicitly. 
For $2\leq p\leq \rho$ we may take the parallelogram with vertices at $0,P,P',P+P'$ where
\begin{equation}\label{par1}
P=(1,0)\;\;\mathrm{and}\;\; P'=\big(\tfrac{1}{2},\tfrac{1}{2}(2^p-1)^{\frac{1}{p}}\big).
\end{equation}
For $1\leq p\leq 2$ or $\rho\leq p\leq\infty$ we may take
\begin{equation}\label{par2}
P=\big(2^{-\frac{1}{p}}(1-\t_p)^{-1},-2^{-\frac{1}{p}}\t_p(1-\t_p)^{-1}\big)\;\;\mathrm{and}\;\;P'=\big(2^{-\frac{1}{p}},2^{-\frac{1}{p}}\big)
\end{equation}
where again $0<\t_p<\frac{1}{2}$ solves $\t_p^p+1=2(1-\t_p)^p.$
Except when  $p=1,2$ or $ \infty$ these parallelograms are unique up to obvious symmetries. 
When $p=1,2$ or $\infty$  there are infinitely many essentially different minimizing parallelograms.
They are easily parameterized. 
For example, when $p=2$ all are obtained by rotating the standard hexagonal lattice coming from (\ref{par1}).

\bigskip
 Minkowski's method can be restated as saying that $3\Delta(\B)$ is the minimal area of an affinely regular symmetric  hexagon inscribed in $\B$.
 A useful alternative due to Reinhardt \cite{Rei} is that  $4\Delta(\B)$ is the minimum area of a symmetric convex circumscribed hexagon  (see also \cite[p.~239]{Cas} or \cite[Thm 2 p.~243]{GL}).
Using this fact, that he also found independently,  Mahler  \cite{Mah1}  computed $\Delta(\B^{\mathrm{oct_1}})=\sqrt{2}-\half$ for the regular octagon from Figure \ref{normfig}. 
Thus we also have $\Delta(\B^{\mathrm{oct_2}})=\frac{1}{8} \left(3 \sqrt{2}+2\right)$, obtained by scaling.


\section{ A Minkowski-type algorithm}
In this section we will prove Theorem   \ref{genl}. First we give a needed definition.
 A {\it minimal basis} for  a lattice $L\subset \R^2$ with respect to a norm $F$ is a $\Z$-basis $\{P,P'\}$  for $L$  with the property that  \[F(P)=F(P')=\min_{\substack{P_0\in L\setminus\{0\}}}F(P_0).\]   
 
 For $\a\in\R$  let 
\begin{equation}\label{lal}
 L_\a=(1,-\a)\Z+(0,1)\Z.
 \end{equation}
Obviously $L_\a$ has determinant one. 
 
 To prove Theorems \ref{genl} and \ref{tcf} we require an algorithm that constructs a sequence of points $P_n\in L_\a$  and positive numbers $t_m$ 
 such that $\{P_{m-1},P_m\}$ gives a minimal basis for $L_\a$ with respect to $F_{t_m}.$ We also want $P_{m-1}$ to have the smallest norm $\|P_{m-1}\|_{t}$ among non-zero points in $L_\a$  for any $t\in (t_{m-1},t_{m})$.
 We will start with $P_{-1}=(0,1)$ and $P_{0}=(1,-\a)$.
Roughly speaking,  given $P_{m-1}$, to find the new point $P_m$ and the associated  $t_m$, we simultaneously expand $\B$ in the $x$-direction while shrinking in the $y$-direction
in a such a way that $P_{m-1}$ remains on its boundary until we encounter $P_m$.  
We then repeat this procedure starting with $P_m$ (see Figure~{\ref{proced}).
Our algorithm will produce pairs of lattice points in $L_\a$ that are linearly independent over $\R$ and on a ball for which $L_\a$ is admissible. First we need to know that they give a basis for $L_\a.$

 \begin{lemma}\label{lb}
  Let $\a\in \R$ be irrational and $F$ be a fixed strongly symmetric norm. 
Suppose that  $P,P'\in L_\a$ lie on the boundary of  $\B_t$  for some $t>0$ and are linearly independent over $\R$. If $L_\a$ is admissible for $\B_t$   then $\{P,P'\}$ 
gives a $\Z$-basis for $L_\a$.
  \end{lemma}
\begin{proof}
Consider the sublattice $P\Z+P'\Z$ of $L_\a$ generated by these lattice points.  By Minkowski's first convex body theorem its index in $L_\a$ can only be 1 or 2.
In the latter case suppose that  $P=aQ+bQ'$ and $P'=cQ+dQ'$ 
where  $L_\a=Q\Z+Q'\Z$, so
 $|ad-bc|=2.$ If $a$ were even and $c$ odd we would have that $b$ is even and so $\tfrac{1}{2} P=(\frac{a}{2})Q+(\frac{b}{2})Q'$ would be a non-zero point in $ \B_t\cap L_\a$.  A similar argument  disallows $c$ being even and $a$ odd. Thus $a$ and $c$ are either both even or both odd.
Similarly $b$ and $d$ are either both even or both odd. In any case 
\[\tfrac{1}{2}(P+P')=(\tfrac{a+c}{2})Q+(\tfrac{b+d}{2})Q'\;\;\;\mathrm{and}\;\;\; \tfrac{1}{2}(P-P')=(\tfrac{a-c}{2})Q+(\tfrac{b-d}{2})Q'\]
are  distinct points of $L_\a$. As $\B_t$ is convex they must lie on the boundary of $\B_t$. It follows that $\B_t$ must be a
parallelogram and strong symmetry implies it is a stretched ball for either the 1-norm or the sup norm.  As the corners and midpoints of the sides are lattice points 
we would have to have that $L_\a$ contains points of the $x$-axis, i.e. $\a$ would be rational.
\end{proof}

In the next lemma we make the whole process precise.
Clearly to represent any $L_\a$ we may assume that $\a\in (-\frac 12,\frac 12].$
  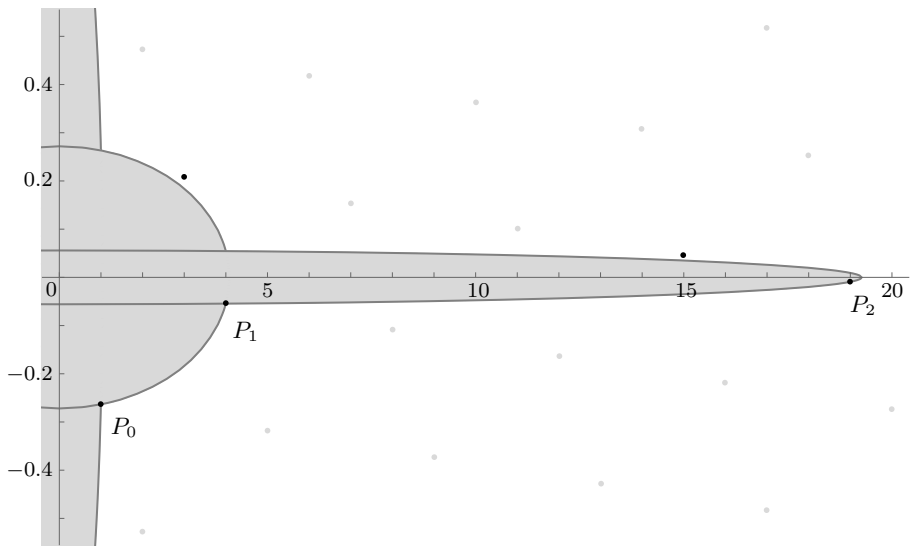
\begin{figure}[h]\label{proced}
    \centering
    \begin{overpic}[width=0.7\textwidth]{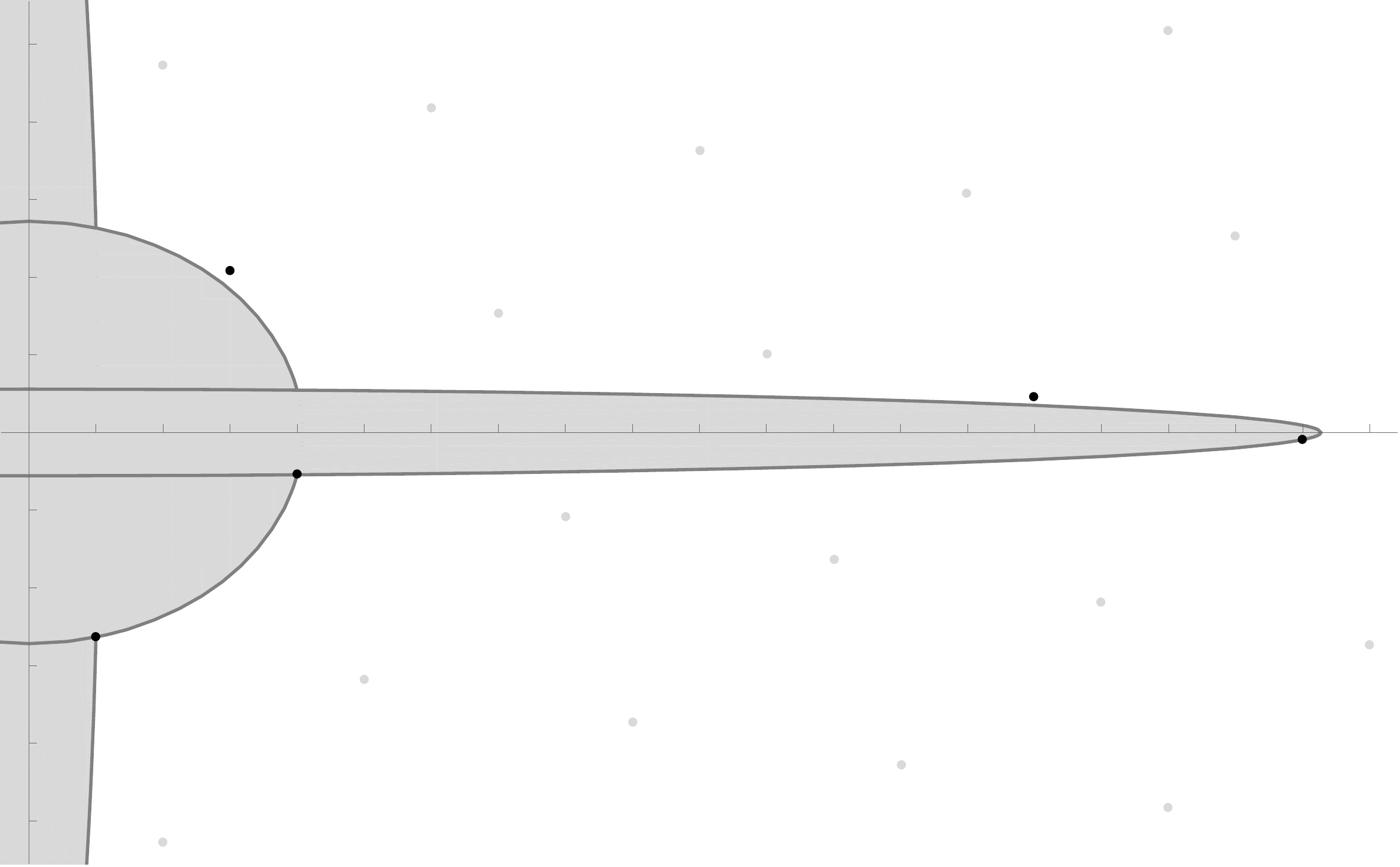}
    \tiny
    \put(25.4,28.7){$5$}
    \put(49,28.7){$10$}
    \put(72.9,28.7){$15$}
    \put(96.6,28.7){$20$}
    \put(0.5,28.9){$0$}
    \put(-2,41.5){$0.2$}
    \put(-2,52.5){$0.4$}
    \put(-4,19.2){$-0.2$}
    \put(-4,8.2){$-0.4$}
    \scriptsize
    \put(8,13){$P_0$}
    \put(22,24){$P_1$}
    \put(93,27){$P_2$}
    \end{overpic}
    \caption{The lattice $L_\alpha$ for $\alpha = \frac{1}{1+}\,\frac{1}{2+}\,\frac{1}{1+}\,\frac{1}{3+}\,\frac{1}{1+}\,\frac{1}{4+}\cdots$. Dark lattice points correspond to the regular convergents of $\alpha$. The points $P_0=(1,-\alpha), P_1 = (4,1-4\alpha)$, and $P_2 = (19, 5-19\alpha)$ give best approximations for the $2$-norm.}
    \label{fig:alg}
\end{figure}

\begin{lemma}\label{alg1}
Fix a strongly symmetric norm $F$
and an  irrational  $\a\in (-\frac 12, \frac 12)$. There is a  sequence $1=t_{-1}\leq t_0< t_1< t_2< \ldots$ tending to $\infty$ and for each $m=-1,0,1,\dots $  there is a $P_m =(x_m,y_m)\in L_\a$ 
with the following properties. For each $m\geq 0$
\begin{enumerate}[label=(\roman*)]
\item $ x_{m}>x_{m-1}$ and $ |y_m|<|y_{m-1}|$,
\item  $\{P_{m-1},P_m\}$ gives a minimal basis  for $L_\a$ with respect to $F_{t_m}$,
	\item 	for any $t\in (t_{m-1},t_{m}) $  there is  no $P'=(x',y')\in L_\a$ different from $P_{m-1}$ with $x'>0$  and
		\[
		F_t(P') \leq F_t(P_{m-1}).
		\]

	\end{enumerate}
 \end{lemma}
\begin{proof}

Consider for $P=(x,y)  \in L_\a$ the ball
\begin{equation}\label{BB}
	\B(P,t)\defeq\{P'\in \R^2 ; F_t(P')< F_t(P)\},
\end{equation}
for which
$\mathrm{area}\,\B(P,t)=F^2_t(P)\,\mathrm{area}\,\B.$
Suppose that $L_\a$ is admissible for $\B(P,t)$.
Now by Lemma \ref{l2}
\[
F_t(P)=F(t^{-1}x,ty)\geq F(0,ty)=t|y|.
\]
Thus, as long as $y\neq 0$, by Minkowski's first convex body theorem there will be a maximal $t'\geq t$ for which $L_\a$ is admissible for  $\B(P,t')$.
For any of the resulting $P'\neq -P$ with $F_{t'}(P')=F_{t'}(P)$, we have by Lemma \ref{lb} that $\{P,P'\}$ gives a minimal basis for $L_\a$ with respect to $F_{t'}.$

Let $P_{-1}=(0,1)$. Then  $L_\a$ is admissible for $\B(P_{-1},1)$. Let $t_0\geq 1$ be  maximal for which $L_\a$ is admissible for  $\B(P_{-1},t_0)$.
Our assumption that $\a\in (-\frac 12,\frac 12)$ implies that 
we can take $P_0 = (1,-\alpha)$ as a solution to $F_{t_0}(P_0)=F_{t_0}(P_{-1})$.

Now $L_\a$ is admissible for $\B(P_{0},t_0)$ and so we find $t_1>t_0$ maximal so that $L_\a$ is admissible for  $\B(P_{0},t_1)$.
That $t_1>t_0$ with strict inequality is assured by our choice of $P_0$. 
Among the finitely many $P'=(x',y')\in L_\a$ with $F_{t_1}(P')=F_{t_1}(P_0)$ there will be unique one   with maximal $x'$ since $\a$ is irrational.
We let  $P_1=(x_1,y_1)$ be this point.  Clearly $x_1>x_0$ and $|y_1|<|y_0|$.

We continue this process to construct $t_m$ and $P_m$.
That we have $t_m>t_{m-1}$ is guaranteed by choosing among the new points on the boundary the one with maximal $x$-coordinate.
From the form of $L_\a$, where $\a$ is irrational, it follows  that $x_{m}>x_{m-1}$ and $ |y_{m-1}|>|y_{m}|>0$ for each $m \geq 0$
  and that this process never terminates.

We will have all the stated properties of $t_m$ and $P_m$ once  we show  that $t_m\rightarrow \infty$. 
We have by Lemma \ref{l2} and Minkowski's first convex body theorem again that for each $m \geq 0$
\[
x_{m}t_m^{-1} =F(x_{m}t_m^{-1},0)\leq  F(x_{m}t_m^{-1},y_{m}t_m)=F_{t_m}(P_m)\leq 2(\mathrm{area}\,\B)^{-\frac 12}.
\]
Thus $t_m\gg x_m \rightarrow \infty$ as $x_m>x_{m-1}$ are integers.
\end{proof}

\subsection*{Proof of Theorem \ref{genl}}

Observe that $\B(P_m,t_m)$ as defined by (\ref{BB}), with $P_m$ and $t_m$ from Lemma \ref{alg1},  contains a parallelogram of area 2
since $\{P_{m-1},P_m\}$ is a minimal basis for $L_\a$ with respect to $F_{t_m}.$
Therefore
\[
\Delta F_{t_m}^2(P_m) \geq 2\Delta(\mathrm{area}\,\B)^{-1}\geq \tfrac{1}{2},
\]
where to get the second inequality we have  applied Minkowski's bound 
(\ref{mi2}). Since $t_m \rightarrow \infty$ as $n \rightarrow \infty$ 
we have that
\begin{equation}\label{dellim}
 \delta_F(\a)\geq \Delta \limsup_{m \rightarrow \infty}F^2_{t_m}( P_m) \geq \tfrac{1}{2},
 \end{equation}
thus proving Theorem \ref{genl}.
\qed

\section{The continued fraction}\label{CF}

Next we relate to each other  the two definitions of best approximation given in and below Definition \ref{def1}.
\begin{lemma}\label{bap}
If the fraction $\frac pq $ with $q>0$ is a best approximation of an irrational $\a$ with respect to a strongly symmetric norm $F$
 then it is a best approximation in the usual sense. Conversely,  if $\frac pq $ with $q>0$ is a best approximation of an irrational $\a$ 
 in the usual sense it is a best approximation with respect to the sup-norm. 
\end{lemma}
\begin{proof}
Suppose that there is a $t> 1$ such that $\frac rs \neq \frac pq$ with $s>0$ implies
\[
F_t(q,p-\a q) < F_t(s,r-\a s).
\]
If $s\leq q$ then $|r-\a s|> |p-\a q|$ by Lemma \ref{l2}.

Conversely, suppose that $\frac rs\neq \frac pq$ and $0<s\leq q$ implies that $|p-\a q|< |r-\a s|$.  Choose $t$ such that $t^{-1}q=t|p-\a q|$
and note that such a $t > 1.$

If $0<s\leq q$ and $\frac rs\neq \frac pq$   then 
\[
\sup(q t^{-1},|p-\a q |t)=t|p-\a q|< t|r-\a s|\leq \sup(st^{-1},|r-\a s|t).
\]
If $s> q$ then
\[
\sup(q t^{-1},|p-\a q|t)=t^{-1}q<t^{-1}s\leq \sup(st^{-1},|r-\a s|t).
\]
This finishes the proof.
\end{proof}

\subsection*{ Proof of Theorem \ref{tcf}}

 For any  $\a\in \R$ write $\a =\a'+a_0$, where $a_0 \in \Z$ and $\a'\in (-\frac12 , \frac12].$
Suppose that $\a$ is irrational. In the notation of Lemma \ref{alg1} (taking there $\a=\a'$)
for $m \geq -1$  write $P_m=(x_m,y_m)$.
For $m\geq 0$ define \[g_m=\pMatrix{x_m}{y_m}{x_{m-1}}{y_{m-1}}\]
and set $g_{-1}=\pmatrix{0}{1}{1}{-\a}$.
We know by Lemma \ref{alg1} that for each $m\geq 1$ there is a positive integer $a_m$ and $\ep_n=\pm 1$ so that
\begin{equation}\label{gn1}
g_m= \pMatrix{a_m}{\ep_m}{1}{0}g_{m-1}.
\end{equation}
This also holds for $m=0$ if we set $\ep_0=1.$ Clearly for $m \geq 0$
\begin{equation}\label{lam}
	\g_m\defeq\det{g_m}=(-1)^{m}\ep_1\cdots\ep_m.
\end{equation}

The numerator $p_m$ and denominator $q_m$ of the convergents 
\[
\frac{p_m}{q_m}=a_0+\frac{\ep_1}{a_1+}\,\frac{\ep_2}{a_2+}\cdots\frac{\ep_m}{a_m}
\]
of our continued fraction  are determined recursively for $m \geq 0$ through
\begin{alignat}{3}\label{r1}
p_m&=a_m p_{m-1}+\ep_mp_{m-2}, \qquad & p_{-1}&=1, \quad & p_{-2}&=0,
\\ \label{r2}q_m&=a_m q_{m-1}+\ep_m q_{m-2}, & q_{-1}&=0, & q_{-2}&=1.
\end{alignat}
It is easy to see that for $m \geq 0$
\begin{equation}\label{cf1}
 \pMatrix{p_{m}}{q_{m}}{p_{m-1}}{q_{m-1}}=\pMatrix{a_m}{\ep_m}{1}{0}\pMatrix{a_{m-1}}{\ep_{m-1}}{1}{0}\cdots \pMatrix{a_0}{1}{1}{0}.
\end{equation}
By  (\ref{gn1}) and (\ref{cf1}) for each $m\geq 0$
we get that
\begin{align}\label{gn}
g_m=\pMatrix{x_m}{y_m}{x_{m-1}}{y_{m-1}}=\pMatrix{q_m}{p_m-\a q_{m}}{q_{m-1}}{p_{m-1}-\a q_{m-1}}.
\end{align}

By Lemma \ref{alg1}  the basis of rows of $g_m$ is minimal for the norm $F_{t_m}$. 
Choose any $t\in(t_m,t_{m+1})$. By (iii) of Lemma \ref{alg1}
 for any $\frac rs \neq \frac{p_m}{q_m}$ we have
\begin{equation}\label{deteq}
F_t(q_m,p_m-\a q_m)
<
F_t(s,r-\a s). 
\end{equation}
Conversely, suppose that $\frac rs$ is a best approximation to $\alpha$ with respect to $F$, and write $Q=(s,r-\alpha s)$.
Then for some $t> 1$, we have $F_t(Q)\leq F_t(P_m)$ for all $m\geq 0$. 
By  Lemma~\ref{alg1} it cannot happen that $t\in (1,t_0)$ since in that case we would have to have $Q=(0,1)$ and so $s=0$. Also we cannot have that $t=t_m$ for any $m\geq 0.$
On the other hand, if $m$ is such that $t\in (t_m, t_{m+1})$, then by Lemma~\ref{alg1} we have $Q=P_m$.
It follows that the convergents are precisely the best approximations to $\alpha$ with respect to $F$.

That  the continued fraction converges to $\a$ now follows from the first statement of  Lemma~\ref{bap} and Lagrange's theorem mentioned below Theorem \ref{tcf},  since they imply 
 that each convergent of our continued fraction is a convergent of the regular continued fraction.
It remains to show that it is semi-regular. Since  $\a$ is irrational  we need only show that 
$\ep_{m+1}+ a_{m}\geq 1$ for all $m\geq1$.
This will follow once we relate the $\mu_m,\nu_m$ from (\ref{unvn}) to the points $P_m=(x_m,y_m)$,
which is also needed 
to prove our generalization of  (\ref{Del}). 
 \begin{lemma}\label{l1}
For $x_m,y_m$ from (\ref{gn}) and $\mu_m,\nu_m$ from (\ref{unvn})  we have for $m\geq 0$ that
\begin{equation}\label{thi}
\mu_m=-\frac{y_m}{y_{m-1}}\;\;\;\mathrm{and}\;\;\;\;\nu_m=\frac{x_{m-1}}{x_m}.
\end{equation}
\end{lemma}
\begin{proof}
The proof is an adaptation to more general continued fractions of standard arguments used for regular continued fractions (see \cite{Sch1}).

To start with, by  (\ref{gn}) 
\begin{equation}\label{xs}
-\frac{y_m}{y_{m-1}}=\frac{-p_m+\a q_m}{p_{m-1}-\a q_{m-1}}.
\end{equation}
By (\ref{cf1}) and (\ref{lam}) we have
\begin{equation}\label{det}
q_{m+1}p_{m}-p_{m+1} q_{m}=\g_{m+1}.
\end{equation}

Together with (\ref{r1}) and (\ref{r2}), this yields the following formal identity between rational functions with variables $a_1,\dots, a_{m+1}$:
\begin{equation}\label{lehe}
p_m-q_m\, \frac{p_{m+1}}{q_{m+1}}= \frac{\g_{m+1}}{a_{m+1}q_m+\ep_{m+1}q_{m-1}}\;\;\;\;\mathrm{where}\;\;\;\;\frac{p_{m+1}}{q_{m+1}}=a_0+\frac{\ep_1}{a_1+}\,\frac{\ep_2}{a_2+}\cdots\frac{\ep_{m+1}}{a_{m+1}}.
\end{equation}
The $m^{th}$ {\it complete quotient} $\a_m$ of the expansion $\a=\frac{\ep_1}{a_1+}\,\frac{\ep_2}{a_2+}\cdots
$
is defined recursively  by $\a_0=\a$ and for $m \geq 0$ through
\begin{equation*}
\a_{m+1}=\frac{\ep_{m+1}}{\a_m-a_m}.
\end{equation*}
It follows that for $m \geq 0$ we have
\begin{equation}\label{cq}
\a=a_0+\frac{\ep_1}{a_1+}\,\frac{\ep_2}{a_2+}\cdots\frac{\ep_{m+1}}{\a_{m+1}}.
\end{equation}
By (\ref{lehe}) upon setting the variable $a_{m+1}=\a_{m+1}$ and using (\ref{cq}) we derive that
\[
p_m-q_m\a=\frac{\g_{m+1}}{\a_{m+1}q_m+\ep_{m+1}q_{m-1}}.
\]
Next solve this equation for $\a_{m+1}$ and use (\ref{det}) with $m$ in place of $m+1$ to get
\begin{equation}\label{anp}
\a_{m+1}=\frac{\ep_{m+1}(-p_{m-1}+q_{m-1}\a)}{p_m-q_m\a}.
\end{equation}
From  (\ref{cq})  we have 
\begin{equation}\label{alp}
\a_{m+1}= a_{m+1}+\frac{\ep_{m+2}}{a_{m+2}+}\,\frac{\ep_{m+3}}{a_{m+3}+}\cdots.
\end{equation}
so by (\ref{unvn}) 
\begin{equation}\label{umal}
\mu_m=\frac{\ep_{m+1}}{\a_{m+1}}.
\end{equation}
The first formula of (\ref{thi}) now follows from  (\ref{xs}) and (\ref{anp}).

To prove the second formula of (\ref{thi}) start with $\frac{q_{m-1}}{q_m}=\frac{x_{m-1}}{x_m}$ from (\ref{gn}).
By (\ref{unvn}) we have that $v_0=0$ while for $m \geq 0$
\[
\nu_{m+1}=\frac{1}{a_{m+1}+\ep_{m+1}\nu_m}.
\]
Using (\ref{r2}) we see that $\frac{q_{m-1}}{q_m}$ satisfies the same recurrence.
\end{proof}

We now finish the proof of Theorem \ref{tcf} by showing that our expansion is semi-regular.
Suppose that we had $\ep_{m+1}=-1$ and $a_{m}=1$ for some $m \geq 1.$
We would then have from (\ref{alp}) that $\a_{m}<1$
and so from (\ref{umal}) that 
\[|\mu_{m-1}|=|\tfrac{y_{m-1}}{y_{m-2}}|>1,\]
which is impossible.
This completes the proof of Theorem \ref{tcf}.\qed

\section{A formula for $\d_F(\a)$}\label{form}
We will deduce Theorem \ref{t6} from a formula for $\d_F(\a)$ for any strongly symmetric norm $F$ given in terms of the quantities $\mu_m,\nu_m$.
As usual, we may identify the space of all  lattices of determinant one with $ \G\backslash G$ where $G=\SL(2,\R)$ and $\G=\SL(2,\Z)$
by means of 
\begin{equation}\label{lofq}
g=\pmatrix{x}{y}{x'}{y'}\mapsto L(g)\defeq(x,y)\Z+(x',y')\Z.
\end{equation}
Let $\mathcal{D}$ be the set of $g=\pmatrix{x}{y}{x'}{y'}\in G$  such that
\begin{align}
&F(x,y)=F(x',y')\;\;\mathrm{and}\\
&0\leq x'<x \;\;\mathrm{and}\;\; |y|<y'.
\end{align}
For $g\in \mathcal{D}$ let $F(g)=F(x,y).$
\begin{lemma}\label{cont}
 The map $\Phi:\mathcal{D}\rightarrow (-1,1)\times[0,1)$  given by
\[
\Phi \pmatrix{x}{y}{x'}{y'} =(-\tfrac{y}{y'},\tfrac{x'}{x})
\]  is a continuous bijection.
\end{lemma}
\begin{proof}
The  inverse of $\Phi$ is given by 
\begin{equation}\label{inver}
(u,v)\mapsto \tfrac{1}{\sqrt{1+uv}}\pmatrix{t^{-1}}{-ut}{t^{-1}v}{t}.
\end{equation}
By Lemma \ref{lll} we see that  $t=t(u,v)>0$ exists and is uniquely determined by the condition $F_t(1,-u)=F_t(v,1).$
\end{proof}
The function 
\begin{equation}\label{Duv}
D_F(u,v)\defeq F^2(\Phi^{-1}(u,v))
\end{equation}
 is easily seen to be continuous on $(-1,1)\times[0,1)$.

The following is our generalization of the formula (\ref{Del}).
\begin{lemma}\label{tl6}
Fix a $F$ strongly symmetric norm.
For any irrational $\a$ whose continued fraction associated to the norm 
is (\ref{srcf1}) we have that
\[
\d_F(\a)=\limsup_{ m\rightarrow \infty}\D\,D_F\big(\mu_m,\nu_m \big),\]
where $\mu_m,\nu_m$ are given above in (\ref{unvn}). 
\end{lemma}
\begin{proof}
Fix an $m$ and write as before $P_m=(x_m,y_m)$.   Let
\[
\Phi^{-1}(\mu_m,\nu_m)=\pmatrix{x}{y}{x'}{y'}
\]
where $0\leq x'<x$ and $  |y|<y'$.
Recall that by (i) of Lemma \ref{alg1} we know that
\[
0\leq x_{m-1}<x_{m}\quad \text{ and } \quad |y_m|<|y_{m-1}|.
\]
 Lemmas  \ref{cont} and \ref{l1} now imply that
\[x=t_m^{-1}x_m,\;\;x'=t_m^{-1}x_{m-1}, \;\; y=\g_m t_m y_m, \;\;y'=\g_mt_m y_{m-1},\]
where $\g_m=\pm 1$ was defined in (\ref{lam}).  Note that in this case   $\g_m=\sgn{y_{m-1}}.$
  By strong symmetry of the norm we have  
 $F_{t_m}(x,y)=F_{t_m}(x',y')=F_{t_m}(P_m)$.
Hence
\[
F^2_{t_m}(P_m)=F^2(\Phi^{-1}(\mu_m,\nu_m))=D_F(\mu_m,\nu_m).
\]

Now we need to show that
\begin{equation}\label{dell}
 \delta_F(\a)= \Delta \limsup_{m \rightarrow \infty} F^2_{t_m}(P_m).
 \end{equation}
For $t\geq 1$ let $m(t)$ be such that $t_{m(t)}\leq t \leq t_{m(t)+1}$.
By Lemma \ref{alg1} we have that 
\begin{equation*}
 \delta_F(\a)\leq  \Delta \limsup_{t \rightarrow \infty} \|F_{m(t)}\|^2_{t},
 \end{equation*}
and by Lemma~\ref{ll2} we have that
\begin{equation*}
	F_t(P_m) \leq \max\left( F_{t_m}(P_m),F_{t_{m+1}}( P_m) \right) \quad \text{ if } \quad t_m\leq t\leq t_{m+1}.
\end{equation*}
Now apply the first inequality in (\ref{dellim}) to establish (\ref{dell}) and therefore
finish the proof of Lemma \ref{tl6}.
\end{proof}

\subsection*{Proof of Theorem \ref{t6}}

To conclude formula (\ref{delta}) from Lemma \ref{tl6},
first observe that for the $p$-norm with  $1\leq p<\infty$ we have from (\ref{inver}) that for $(u,v)\in (-1,1)\times[0,1)$ the value of $t$ that makes   the rows of 
$\Phi^{-1}(u,v)$ have the same norm $F_t$ is given by 
\[
t=\Big(\frac{1-v^p}{1-|u|^p}\Big)^{\frac{1}{2p}}.
\]
The corresponding value of $ D_{F^{\langle p\rangle}}(u,v)$ from (\ref{Duv}) is 
\begin{align*}
D_{F^{\langle p\rangle}}(u,v)=(1+uv)^{-1}\Big(\big(\tfrac{1-v^p}{1-|u|^p}\big)^{-\frac12}+|u|^p\big(\tfrac{1-v^p}{1-|u|^p}\big)^{\frac12}\Big)^{\frac2p}\\
=(1+uv)^{-1}\Big( \tfrac{(1-|u|^pv^p)^2}{(1-|u|^p)(1-v^p)}\Big)^{\frac1p}=D_p(u,v),
\end{align*}
giving  (\ref{delta}).  The case $p=\infty$ is immediate.
This completes the proof of Theorem \ref{t6}.\qed

\section{$\mathcal S$-expansions} \label{sec:s-exp}

To prove Theorem \ref{sexp} we want to characterize in terms of the norm those convergents of the regular continued fraction 
of an irrational $\a$ that are also convergents of the continued fraction of $\a$ associated to a strongly symmetric norm $F$. We will use the notation and results of  Lemma \ref{alg1}.
Write $P_m=(q_m,p_m-\a q_m )$ for points coming from this norm  with corresponding $t_m$ and let $Q_n=(s_n,r_n-\a s_n )$ be the points coming from the convergents  of the regular continued fraction of $\a$.
Furthermore, the partial quotient $b_n$ is associated to $Q_n$ while $a_m$  is associated to $P_m$.

\begin{lemma}\label{char0}
For a fixed $n\geq 0$ there are  integers $c_\ell$ and $d_\ell$ with $c_\ell>0$ and $d_\ell\geq 0$ so that for each $\ell\geq 0$ 
\begin{equation*}
Q_{n+\ell}=c_\ell Q_n+d_\ell Q_{n-1}
\end{equation*}
where $c_\ell\geq d_\ell $ for all $\ell\geq 0$, while for $\ell \geq 2$ we have
\[
c_\ell\geq d_\ell +1.
\]

 \end{lemma}
\begin{proof}
The integers $r_n,s_n$ are determined recursively for $n\geq 0$ by
\begin{alignat}{3}
\label{r00}
r_n&=b_n r_{n-1}+r_{n-2}, \qquad &r_{-2}&=0, \quad &r_{-1}&=1,
\\
\label{r0}
s_n&=b_n s_{n-1}+s_{n-2}, \qquad &s_{-2}&=1,\quad &s_{-1}&=0.
\end{alignat}

It is easy to check using (\ref{r00}) and (\ref{r0}) that $c_\ell$ and $d_\ell$ satisfy  for fixed $n$ and $\ell\geq 1$ the recurrence relations
\begin{alignat}{3}
\label{eq:cl-rec}
c_{\ell+1}&=b_{n+\ell+1}  c_{\ell}+c_{\ell-1},\qquad &c_{1}&=b_{n+1},\quad &c_{0}&=1,
\\ 
\label{eq:dl-rec}
d_{\ell+1}&=b_{n+\ell+1} d_{\ell}+d_{\ell-1},\qquad &d_{1}&=1,\quad & d_{0}&=0.
\end{alignat}
The claim of the lemma follows from a straightforward  inductive argument.
  \end{proof}

The following result will be used to characterize those convergents of the regular continued fraction that 
occur as convergents in the continued fraction associated to the norm.
\begin{lemma} \label{lem:ell=1}
For $m \geq 1$ let $n$ and $\ell$ be such that $P_{m-1}=Q_{n-1}$  and $P_{m}=Q_{n+\ell}$.
Then 
\begin{enumerate}[label=(\roman*)]
	\item $\ell\in \{0,1\}$.
	\item There is a unique  $t\geq 1$ such that
		$F_t(Q_n)=F_t(Q_{n-1})$,
	and $\ell=1$ if and only if 
	\[
		F_t(Q_{n}+Q_{n-1})\leq F_t(Q_{n}).
	\]
	If this holds we have that $b_{n+1}=1.$
	\item  $Q_0=P_0$ if and only if $a_0=b_0.$
\end{enumerate}
\end{lemma}
\begin{proof}

We know that $P_{-1}=Q_{-1}$ and that for each $m \geq 1$ we have $P_{m-1}=Q_{n-1}$ for some $n$ and $P_{m}=Q_{n+\ell}$ for some $\ell\geq 0.$ We can check directly that $Q_0=P_0$ if and only if $a_0=b_0.$

By Lemma \ref{char0} for $\ell \geq 0$ we have
\begin{equation}\label{inn}
c_\ell Q_n =P_{m}-d_\ell P_{m-1}.
\end{equation}
By Lemma \ref{alg1}
we have $F_{t_m}(Q_n) \geq F_{t_m}( P_m)=F_{t_m}(P_{m-1})$ and hence
\begin{equation}\label{ini}
c_\ell F_{t_m}(P_m) \leq F_{t_m}(P_{m}-d_\ell P_{m-1})< F_{t_m}(P_{m})+d_\ell F_{t_m}(P_{m-1}).
\end{equation}
 by Lemma \ref{equal}.
Thus  we have 
\begin{equation} \label{eq:cldl}
c_\ell < d_\ell+1.
\end{equation}
so by Lemma \ref{char0} we  have 
 that either $\ell=0$ or $\ell=1.$

Now by Lemma \ref{lll} applied to the norm $F_{t_m}$ and using that
\[
F_{t_m}(Q_n)\geq F_{t_m}(Q_{n-1}),
\]
there is a $t\geq 1$ (indeed $t\geq t_m$) so that
\[
 F_t(Q_n)=F_t(Q_{n-1}).
\]

In case $\ell=1$ we have $b_{n+1}=1$ by \eqref{eq:cldl} and \eqref{eq:cl-rec}--\eqref{eq:dl-rec}. 
By (\ref{inn}) we have that
\[
 Q_n+Q_{n-1} =P_{m}=Q_{n+1},
\]
so by Lemma \ref{alg1}
we must have
\[
F_t(Q_n+Q_{n-1})=F_t(Q_{n+1})=F_t(P_m) \leq  F_t(Q_n).
\]

If $\ell=0$ we have $Q_{n-1}=P_{m-1}$ and $Q_{n}=P_m$ so that $t=t_m$ and 
\[
F_t(Q_n+Q_{n-1})=F_{t_m}(P_m+P_{m+1})>F_{t_m}(P_m)=F_t(Q_{n}),
\]
at least when $m\geq 0$, since then the $x$-coordinate of $P_m+P_{m+1}$ is strictly larger than that of
$P_m$ and so by Lemma \ref{alg1} strict inequality must hold.
\end{proof}

\subsection*{Proof of Theorem \ref{sexp}}

Lemma~\ref{lem:ell=1} gives instructions for obtaining the sequence of convergents ${p_m}/{q_m}$ of $\alpha$ associated to the norm $F$ from the sequence of regular convergents ${r_n}/{s_n}$ of $\alpha$, namely
\begin{equation} \label{eq:sing-rule}
	\text{omit the regular convergent }\mfrac{r_n}{s_n} \ (n\geq 1) \iff F_t(Q_n+Q_{n-1}) \leq F_t(Q_n),
\end{equation}
where $t\geq 1$ is such that $F_t(Q_n)= F_t(Q_{n-1})$, and
\begin{equation} \label{eq:sing-rule-0}
	\text{omit }\mfrac{r_0}{s_0} \iff \lfloor \alpha\rfloor \text{ is not the nearest integer to $\alpha$} \iff \alpha\in [\tfrac 12,1) +\Z.
\end{equation}
We must define a singularization area that encodes both of these instructions.
Let $\mathcal D$ and $\Phi$ be as in Section~\ref{form}.
For each $g=\pmatrix{x}{y}{x'}{y'} \in \mathcal D$ we write
\begin{equation*}
	P = (x,y) \quad \text{ and } \quad P'=(x',y'),
\end{equation*}
and we define
\begin{equation}
	\mathcal S = \Phi\left( \{g\in \mathcal D : F(P + P') \leq F(P)\} \right) \cup \left(\left[\tfrac 12, 1\right)\times \{0\}\right). \label{eq:S-def}
\end{equation}
The portion of $\mathcal S$ that lies on the $u$-axis encodes the rule \eqref{eq:sing-rule-0}.
Suppose that $n\geq 1$ and let $\pmatrix{x}{y}{x'}{y'} = \Phi^{-1}( u_n,v_n)$.
Then Lemmas~\ref{cont} and \ref{l1} imply that $Q_n=(tx,t^{-1}y)$ and $Q_{n-1}=(tx',t^{-1}y')$, with $t$ defined by $F_t(Q_n) = F_t(Q_{n-1})$.
So the condition on the right-hand side of \eqref{eq:sing-rule} is equivalent to $F(P+P')\leq F(P)$.
It follows that \eqref{eq:sing-rule}--\eqref{eq:sing-rule-0} are encoded by the rule
\begin{equation} \label{eq:sing-rule-S}
	\text{omit the regular convergent }\mfrac{r_n}{s_n} \iff ( u_n, v_n) \in \mathcal S.
\end{equation}

It is helpful to have some more concrete information about the set $\mathcal S$.
For a generic norm it is difficult to describe $\mathcal S$ explicitly, so we will relate $\mathcal S$ to the set $\mathcal S_1$, which is easy to describe.
As usual, we denote $\mathcal S$ by $\mathcal S_p$ when $F$ is the $p$-norm.
We have
\begin{equation} \label{eq:S1-def}
	\mathcal S_1 = \left\{ (u,v)\in [\tfrac 12,1)\times [0,1] ;\, v\leq 2-\tfrac 1u \right\},
\end{equation}
as one can see by reducing the system of equations and inequalities
\begin{equation*}
	|x+x'|+|y+y'|\leq |x|+|y| = |x'|+|y'|, \quad x>x'\geq 0, \quad y'>|y|, \quad xy'-yx'=1
\end{equation*}
defining $\mathcal S_1$.
The interior of the set \eqref{eq:S1-def} agrees with the $S$-region given in \cite{Kra2} for Minkowski's diagonal continued fraction.

\begin{lemma} \label{lem:1-norm}
For any strongly symmetric norm $F$ we have $\mathcal S\subseteq \mathcal S_1$.
\end{lemma}

\begin{proof}
Since $\mathcal S$ and $\mathcal S_1$ are closed sets in the induced topology on $[\frac 12,1)\times [0,1]$, it suffices to show that a dense subset of $\mathcal S$ is contained in $\mathcal S_1$.
Suppose that $(u,v)\in \mathcal S$ with $u\notin \Q$ and $v\in \Q$, and write
\begin{equation}
u=\frac{1}{b_{n+1}+}\;\frac{1}{b_{n+2}+}\cdots\;\;\mathrm{and}\;\;v=\frac{1}{b_{n}+}\;\frac{1}{b_{n-1}+}\;\frac{1}{b_{n-2}+}\cdots \frac{1}{b_1}
\end{equation}
for the regular continued fractions of $u$ and $v$.
If we define
\begin{equation*}
	\alpha = \frac{1}{b_{1}+}\;\frac{1}{b_{2}+}\;\frac{1}{b_{3}+}\cdots
\end{equation*}
then $(u,v)=(u_n,v_n)$ for $\alpha$.
Since $(u,v)\in \mathcal S$ we have $Q_{n+1} = Q_{n-1}+Q_n$ in the notation of Section~\ref{sec:s-exp}, and for some $m$ we have $P_{m-1}=Q_{n-1}$ and $P_m = Q_{n+1}$.
Thus
\begin{equation*}
	F_{t_m}(Q_{n-1})= F_{t_m}(Q_{n+1}) \leq F_{t_m}(Q_n).
\end{equation*}
Let $t$ be such that $F^{\langle 1\rangle}_t(P_{m-1})=F_t^{\langle 1\rangle}(P_{m})$, where $F^{\langle 1\rangle}$ denotes the $1$-norm.
By convexity of $F$, the closed stretched ball $\overline{\mathcal B(P_m,t_m)}$ contains the line segment connecting the points $P_{m-1}$ and $P_m$.
This line segment comprises all of the points $P$ in the same quadrant as $Q_{n-1}, Q_{n+1}$ with $x$-coordinate between $x_{m-1}$ and $x_m$, and with  $F^{\langle 1\rangle}_t(P) = F^{\langle 1\rangle}_t(P_m)$.
Since the $x$-coordinate of $Q_n$ is between $x_{m-1}$ and $x_m$ and $Q_n$ is outside the ball $\mathcal B(P_m,t_m)$, we have 
\[ F_t^{\langle 1\rangle}(Q_n)\geq F_t^{\langle 1\rangle}(P_m)= F_t^{\langle 1\rangle}(Q_{n}+Q_{n-1}). \]
By \eqref{eq:sing-rule} and \eqref{eq:sing-rule-S} it follows that $(u,v)\in \mathcal S_1$.
\end{proof}

Lemma~\ref{lem:1-norm}, together with the explicit description \eqref{eq:S1-def}, shows that the set $\mathcal S$ is a singularization area as defined above Theorem \ref{sexp}.
This fact and \eqref{eq:sing-rule-S} together prove Theorem~\ref{sexp}.\qed

We also immediately obtain the following lemma, which we will use several times in the coming sections.

\begin{lemma}\label{line}
For every strongly symmetric norm $F$, there is a neighborhood of the line segment $u=v$ with $u,v\in (0,1)$ that does not intersect $\mathcal S$.
\end{lemma}

We finish this section with a quick proof of our claim \eqref{d2} that
\begin{equation*}
	\delta_2\left(\tfrac 12(-1+\sqrt 3)\right) = 1.
\end{equation*}
The regular continued fraction expansion of $\alpha = \tfrac 12(-1+\sqrt 3)$ is
\begin{equation}
	\alpha = \frac{1}{2+} \; \frac{1}{1+} \; \frac{1}{2+} \; \frac{1}{1+}  \cdots,
\end{equation}
from which it follows that
\begin{equation*}
	u_n = 
	\begin{cases}
		\alpha & \text{ if $n$ is even}, \\
		2\alpha & \text{ if $n$ is odd},
	\end{cases}
\end{equation*}
while $v_{2n} \to 2\alpha$ from below and $v_{2n+1}\to \alpha$ from above.
The region $\mathcal S_2$ comprises those points $(u,v)$ for which $u (2 + v) > 1 + 2 v$.
The points $(u_n,v_n)$ are all outside $\mathcal S_2$, so
the $2$-continued fraction expansion of $\alpha$ is the same as the regular continued fraction and thus $(\mu_n,\nu_n) = (u_n,v_n)$.
Since $D_2(u,v)=D_2(v,u)$, we have
\begin{equation*}
	\delta_2(\alpha) = \Delta_2 \max_{k\in \{0,1\}} \lim_{n\to\infty} D(\mu_{2n+k},\nu_{2n+k}) =  \Delta_2 D(\alpha,2\alpha) = 1.
\end{equation*}

\section{Values of $\d_F(\a)$ for well approximable numbers}\label{bad}

We now prove Theorem \ref{genl2}, which gives the smallest value of $\d_F(\a)$ for  $F$ any strongly symmetric norm
and $\a$ well approximable. 

\begin{lemma}\label{wa}
Suppose that $\a$ is well approximable.
Then $\d_F(\a)\geq \D.$
\end{lemma}
\begin{proof}
By definition, for any $\ep>0$ 
there are arbitrarily large $q>0$ so that for some $p\in \Z$
\[
|\tfrac{p}{q}-\a |<\tfrac{\ep}{q^2}.
\]
For such a $q$ let $t=q$ and note that for any $r,s\in \Z$ with $s>0$
\begin{align*}
F_t(s,r-\a s )=F(t^{-1}s,t(s\a -r))=F(\tfrac{s}{q},q(s\a -r))\\
=F(\tfrac{s}{q},q(s\tfrac{p}{q}-r+\tfrac{\sigma s}{q^2}))=F(\tfrac{s}{q},sp-rq+\tfrac{\sigma s}{q})
\end{align*}
for some $\sigma$ with $|\sigma|\leq \ep.$
By Lemma \ref{l2} if $s\geq q$ we have  \[F(\tfrac{s}{q},sp-rq+\tfrac{\sigma s}{q})\geq F(1,0)=1,\]
while for $0<s<q$  we have $F(\tfrac{s}{q},s p-rq+\tfrac{\sigma s}{q})\geq F(0,1-\ep),$
since  $q\nmid s$.
By the continuity of $F$, for any $\ep'>0$ there is an $\ep>0$
so that $F(0,1-\ep)\geq 1-\ep'$.
It follows that $F_t(s,r-\a s )\geq 1$ and hence that $\d_F(\a)\geq \D.$
\end{proof}

To finish the proof  of Theorem \ref{genl2}, we need to find well approximable $\a$ for which $\d_F(\a)=\D.$
 \begin{lemma}\label{strict}
Suppose that the partial quotients $b_n$ of the regular continued fraction expansion of $\a$ are eventually strictly increasing
with  $n$.
Then 
\[
\d_F(\a)=\D.
\]
 \end{lemma}
\begin{proof}
If the regular partial quotients $b_n$ of $\alpha$ are eventually strictly increasing, then for any $\ep>0$ the points $(u_n,v_n)$ all eventually lie within $\ep$ of the point $(0,0)$.
So by Lemma~\ref{line}, the points $(u_n,v_n)$ are outside $\mathcal S$ for sufficiently large $n$.
Thus 
\begin{equation*}
  \lim_{n\to\infty} (\mu_n,\nu_n) = \lim_{n\to\infty} (u_n,v_n) = (0,0).
\end{equation*}
Finally, $D_F(0,0) = F^2 \left(\pmatrix 1001 \right) = 1$,
therefore $\delta_F(\alpha) = \Delta$.
 \end{proof}

\section{Values of $\d_p(\a)$  for any irrational $\a$}\label{small}

\subsection*{Proof of Theorem \ref{new2}}
Fix $p\in [1,\infty]$.
Throughout the proof let
\begin{equation} \label{eq:alpha-exp}
 \alpha = \frac{\ep_1}{a_1 + } \, \frac{\ep_2}{a_2 + } \, \frac{\ep_3}{a_3 +} \ldots
\end{equation}
denote the $p$-continued fraction expansion of $\alpha \in (0,1)$ and define $\mu_m$ and $\nu_m$ as in \eqref{unvn}.
By Theorem~\ref{t6} it suffices to show that for every $\alpha$ we have
\begin{equation} \label{eq:Dp-cases}
	\limsup_{m\to\infty} D_p(\mu_m,\nu_m) \geq 
	\begin{cases}
		1 & \text{ if } 1\leq p\leq 2, \\
		\tfrac {1}{10}\left(\sqrt 5+5\right) \left( \left(\tfrac 12(\sqrt 5-1)\right)^p+1 \right)^{2/p} & \text{ if }p> 2,
	\end{cases}
\end{equation}
and that there is at least one $\alpha$ for which equality holds.
In both cases the number on the right-hand side of \eqref{eq:Dp-cases} is $\leq 1$.
Since $(1-|u|^pv^p)^2 \geq (1-|u|^p)(1-v^p)$, we have
\begin{equation}
 D_p(u,v) \geq \frac{1}{1+uv}.
\end{equation}
It follows that $D_p(u,v)\geq 1$ for nonpositive $u$, so if $\mu_m\leq 0$ for infinitely many $m$, the inequality \eqref{eq:Dp-cases} holds trivially.
Thus we may assume that the $p$-continued fraction expansion of $\alpha$ has $\mu_m\geq 0$ for all sufficiently large $m$.

\begin{lemma} \label{lem:ineq}
If $0\leq u,v<1$ then
\begin{equation*}
 D_p(u,v) \geq D_p\left( \mfrac{u+v}{2}, \mfrac{u+v}{2} \right),
\end{equation*}
with equality only when $u=v$.
\end{lemma}
\begin{proof}
The inequalities
\begin{align}
 1-(uv)^p \geq 1-\left(\mfrac{u+v}{2}\right)^{2p} \label{eq:ineq} \;\;\;\;\;\text{and}\;\;\;\;
 1+uv \leq 1+ \left(\mfrac{u+v}{2}\right)^{2}
\end{align}
both reduce to $(u-v)^2 \geq 0$.
It remains to show that
\begin{equation*}
 (1-u^p)(1-v^p) \leq \left(1 - \left(\mfrac{u+v}{2}\right)^{p}\right)^2.
\end{equation*}
This inequality is implied by the first inequality of \eqref{eq:ineq} and
\begin{equation*}
 u^p + v^p \geq 2 \left(\mfrac{u+v}{2}\right)^{p},
\end{equation*}
which follows immediately from H\"older's inequality.
\end{proof}

It is convenient to define
\begin{equation*}
 d_p(x) = D_p(x,x) = \frac{(1+x^p)^{2/p}}{1+x^2}.
\end{equation*}
Then
\begin{equation*}
 \frac{\partial}{\partial x} [d_p(x)]^p = \frac{2p(1+x^p)(x^p-x^2)}{x(x^2+1)^{p+1}}.
\end{equation*}
If $1\leq p<2$ then $d_p(x)$ is strictly increasing, so by Lemma~\ref{lem:ineq} we have
\begin{equation}
 \min_{u,v\in [0,1]} D_p(u,v) = \min_{x\in [0,1]}d_p(x) = d_p(0) = 1.
\end{equation}
If $p=2$ then $d_p(x)=1$ for all $x$.
In either case, we can use Lemma~\ref{strict} to find examples of $\alpha$ for which $\delta_p(\alpha)=\Delta_p$.

Suppose that $p>2$, and let $\beta = \frac12(\sqrt 5-1)$.
The sequence $(u_n, v_n)$ associated to the regular continued fraction
\begin{equation*}
	\beta = \frac{1}{1+} \, \frac{1}{1+} \, \frac{1}{1+}\ldots
\end{equation*}
approaches $(\beta,\beta)$ as $n\to\infty$.
By Lemma \ref{line} it follows that the sequence $(\mu_m,\nu_m)$ associated to the $p$-continued fraction of $\beta$ also converges to $(\beta,\beta)$.
Thus, for $p>2$ we have $\delta_p(\beta)=D_p(\beta,\beta)$, which is the number in (\ref{Dp}).

It remains to show that for every $\alpha$ with $\mu_m\geq 0$ for sufficiently large $m$, we have $D_p(\mu_m,\nu_m)\geq D_p(\beta,\beta)$ for infinitely many $m$.
Since $p>2$, the function $d_p(x)$ is strictly decreasing, so by Lemma~\ref{lem:ineq} it suffices to show that
\begin{equation*}
\mu_m+\nu_m \leq \sqrt 5-1 = 1.23606\ldots
\end{equation*}
for infinitely many $m$.

If there are infinitely many $m$ such that $a_{m+1}\geq 5$ then for such $m$ we have $\mu_m\leq \frac 15$ and therefore $\mu_m+\nu_m \leq 1.2$.
So we may suppose that $a_m \leq 4$ for all sufficiently large $m$.
The following lemma covers the remaining cases.

\begin{lemma}
Let $\ell\in \{2,3,4\}$
and suppose that $\ep_m=1$ and $a_m\leq \ell$ for sufficiently large $m$.
If $a_{m}=\ell$ for infinitely many $m$, then
\begin{equation}
	\mu_m+\nu_m < 1.18
\end{equation}
for infinitely many $m$.
\end{lemma}

\begin{proof}
Suppose that $\ep_m=1$ and $a_m\leq \ell$ for $m\geq M$.
Then for any $m\geq M+3$ with $a_{m+1}=\ell$ we have
\begin{align*}
	\mu_m &\leq \frac{1}{\ell+} \, \frac{1}{\ell+} \, \frac{1}{1+} \, \frac{1}{\ell+} \, \frac{1}{1+} \, \frac{1}{\ell+} \, \cdots, \\
	\nu_m &\leq \frac{1}{1+} \, \frac{1}{\ell+1}.
\end{align*}
The lemma now follows from an easy computation.
\end{proof}

This completes the proof of Theorem  \ref{new2}.\qed

\bigskip
We remark that it is sometimes possible  to compute the minimum value of $\d_F(\a)$  for other norms as well.
The composition of strongly symmetric norms is,  up to scaling,
also strongly symmetric (see Lemma \ref{lcom}). 
For instance, the norms $F^{\mathrm{oct_1}}$ and $F^{\mathrm{oct_2}}$ with regular octagonal unit balls mentioned in \S \ref{intro} 
 can be given in terms of compositions of the 1-norm and the sup-norm. Explicitly,
\begin{equation}\label{octn}
F^{\mathrm{oct_1}}(P)= F^{\langle \infty\rangle}(Q) \;\;\;\text{and}\;\;\;\; F^{\mathrm{oct_2}}(P)=(2-\sqrt{2})F^{\langle 1\rangle}(Q)
\end{equation}
where $Q=(\tfrac{1}{\sqrt{2}}F^{\langle 1\rangle}(P),F^{\langle \infty\rangle}(P)).$
These formulas, together with Mahler's computation of the critical determinant of the regular octagon 
recalled at the end of \S \ref{GN} and Lemma \ref{line}, lead to the result referred to at the end of \S \ref{intro}.  The minimum of $\d_F(\a)$ for  $F=F^{\mathrm{oct_1}}$ is $\frac{1}{8} \left(3 \sqrt{2}+2\right), $ which is attained when
\[\a=\sqrt{2}-1=\frac{1}{2+} \, \frac{1}{2+} \, \frac{1}{2+}\ldots.\]
For $F=F^{\mathrm{oct_2}}$ the minimal value is also 
$\frac{1}{8} \left(3 \sqrt{2}+2\right)$, but now this is the value of $\D$ and is attained when $\a=\frac{e-1}{e+1}$,
for instance.



\section{The dynamical system}
The goal of this section is to prove Theorem~\ref{t4}.
We employ the notation of Section~\ref{sec:s-exp}.
Say that $g=\pmatrix{x}{y}{x'}{y'} \in G$ is {\it reduced with respect to the norm} $F$ if $g \in \mathcal{D}$ and
\begin{equation}\label{reduce}
	\left(F(P)\overline {\mathcal B}\right) \cap L(g) = \{0, \pm P, \pm P'\},
\end{equation}
where the overline denotes the closure and $L(g)$ was defined in \eqref{lofq}.
Let $\mathcal{R}$ be the set of all $g$ that are reduced with respect to $F$ and define
$\Omega \subset (-1,1)\times[0,1]$ as
\begin{equation}\label{Om}
\Omega\defeq \Phi(\mathcal{R}) \cup \mathcal A,
\end{equation}
where
\begin{equation*}
  \mathcal A = \overline{\Phi(\mathcal{R})} \cap \left((-\tfrac 12,\tfrac 12)\times\{0\}\right).
\end{equation*}
We will show that $(\mu_n,\nu_n)\in \Omega$ for all $n\geq 0$.

We want to apply the ergodic theory of $\Sc$-expansions as developed in  \cite{Kra1}.
For that we  need to show that $\Omega$ defined by \eqref{Om} coincides with the set 
\begin{equation*}
 	\Omega_{\mathcal S} = \left([0,1)\times [0,1] \setminus (\mathcal S \cup T\mathcal S)\right) \cup (M\circ T) \mathcal S
\end{equation*}
defined in Section~5 of \cite{Kra1}, where
\begin{equation*}
	M(u,v) = \left( \mfrac{-u}{1+u}, 1-v \right), \quad (u,v)\in T\mathcal S.
\end{equation*}

The following equivalent description of reduced matrices is helpful.
\begin{lemma}
A matrix $g\in \mathcal D$ is reduced if and only if
\begin{equation} \label{eq:R-pm}
	\min\big(F(P+P'),F(P-P')\big) > F(P).
\end{equation}
\end{lemma}

\begin{proof}
Clearly a matrix $g$ satisfying \eqref{reduce} also satisfies \eqref{eq:R-pm}.
Suppose $g\in \mathcal D$ satisfies \eqref{eq:R-pm}.
We will show that $F(aP+bP')>F(P)$ for all $(a,b)\in \Z^2\setminus\{(0,0),(0,\pm 1),(\pm 1,0)\}$.
If $|a|=|b|$ then 
\begin{equation*}
	F(aP+bP') = |a|F(P\pm P') > F(P).
\end{equation*}
Otherwise, if $|a|>|b|$, say, then by the reverse triangle inequality
\begin{equation*}
	F(aP+bP') \geq |a|F(P) - |b|F(P') = (|a|-|b|) F(P).
\end{equation*}
This is strictly greater than $F(P)$ if $|a|-|b|\geq 2$.
If $|a|=|b|+1$ then
\begin{equation*}
	F(aP+bP') = F(a(P\pm P') \pm P') \geq |a|F(P+P') - F(P') > (|a|-1)F(P).
\end{equation*}
This completes the proof since $|b|\geq 1$ so $|a|\geq 2$.
\end{proof}

It follows that the set $(-1,1)\times[0,1]$ decomposes as $\Omega \sqcup \mathcal S \sqcup \mathcal S' \sqcup \mathcal S''$, where
\begin{align}
	\mathcal S' &= \Phi\left( \{g\in \mathcal D : F(P - P') \leq F(P) \text{ and } y < 0 \} \right), \notag \\
	\mathcal S'' &= \Phi\left( \{g\in \mathcal D : F(P - P') \leq F(P) \text{ and } y \geq 0  \} \right) \cup \left((-1,\tfrac 12)\setminus \mathcal A\right). \notag
\end{align}
See Figure~\ref{fig:omega} for the case $p=2$.

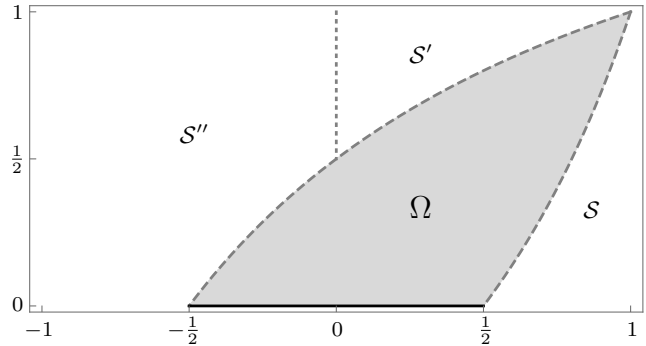
\begin{figure}[h]
   \centering
   \begin{overpic}[width=0.5\textwidth]{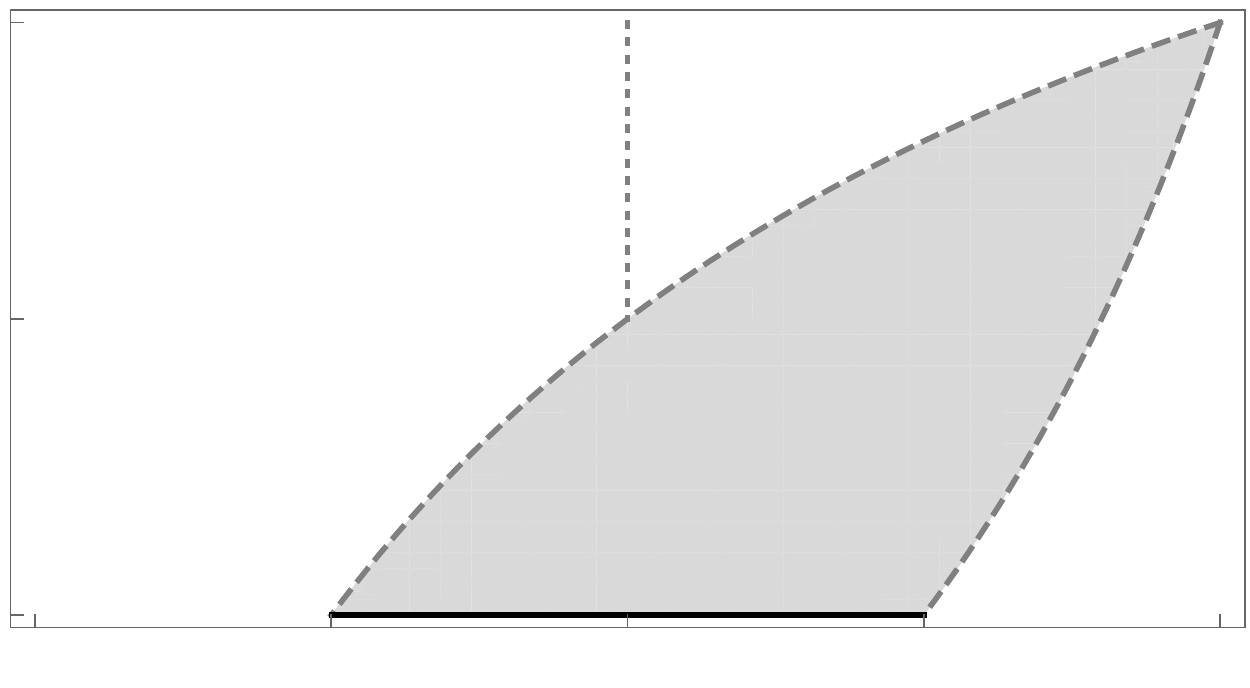}
    \put(62,20){$\Omega$}
    \scriptsize
    \put(90,20){$\mathcal S$}
    \put(62,45){$\mathcal S'$}
    \put(25,32){$\mathcal S''$}
    \tiny
    \put(0,1){$-1$}
    \put(23,1){$-\frac 12$}
    \put(49.5,1){$0$}
    \put(73,1){$\frac 12$}
    \put(97,1){$1$}
    \put(-2,5){$0$}
    \put(-2.5,28.6){$\frac 12$}
    \put(-2,52){$1$}
   \end{overpic}
   \caption{The sets $\Omega$, $\mathcal S$, $\mathcal S'$, and $\mathcal S''$ for $p=2$.}
   \label{fig:omega}
\end{figure}
Since the critical lattices for $F$ are among those for which two basis vectors and their sum all have equal norm, we will refer to the set
\begin{equation} \label{eq:crit}
	\left\{g\in \mathcal{D}; \min\big(F(P+P'),F(P-P')\big) = F(P)\right\}
\end{equation}
as the potentially critical matrices.
The next lemma describes the boundary of $\Omega$ in terms of
the distinguished subset 
\begin{equation*}
	\mathcal P = \left\{ g\in \mathcal D : F(P)=F(P+P') \right\}
\end{equation*}
of the potentially critical matrices.

\begin{lemma} \label{lem:bound}
The part of the boundary of $\Omega$ that lies in $(-1,1)\times [0,1]$ is $\partial \cup \partial' \cup \partial'' \cup \mathcal A$, where
\begin{align*}
	\partial &= \Phi(\mathcal P), \\
	\partial' &= \left\{ \Phi\left(\pmatrix{x+x'}{y+y'}{x\vphantom{x'}}{y}\right) ;\, g\in \mathcal P \right\}, \\
	\partial'' &= \left\{ \Phi\left(\pmatrix{x+x'}{y+y'}{x'}{y'}\right) ;\, g\in \mathcal P \right\}.
\end{align*}
\end{lemma}

\begin{proof}
The boundary of $\Omega$ is $\mathcal A \cup \mathcal C$, where
\begin{equation*}
	\mathcal C = \Phi\left(\left\{g\in \mathcal{D};\, x'>0\text{ and } \min\big(F(P+P'),F(P-P')\big) = F(P)\right\}\right).
\end{equation*}
Clearly $\partial$ is the part of $\mathcal C$ adjacent to $\mathcal S$.
The remaining set, $\mathcal C\setminus \partial$, is the image of the set of $g'\in \mathcal D$ satisfying
\begin{equation*}
	F(Q)=F(Q-Q')<F(Q+Q'), \text{ where } g' = \left(\begin{smallmatrix}Q \\ Q'\end{smallmatrix}\right).
\end{equation*}
If the $y$-coordinate of $Q$ is negative, then $(P,P')=(Q',Q-Q')$ gives an element of $\mathcal P$, and
\begin{equation*}
	\Phi(g') = \Phi\pmatrix{x+x'}{y+y'}{x\vphantom{x'}}{y}.
\end{equation*}
Otherwise $(P,P')=(Q-Q',Q')$ yields an element of $\mathcal P$; in this case
\begin{equation*}
	\Phi(g') = \Phi\pmatrix{x+x'}{y+y'}{x'}{y'}.
\end{equation*}
This completes the proof.
\end{proof}

\begin{figure}[h]
  \centering
  \includegraphics[width=0.38\textwidth]{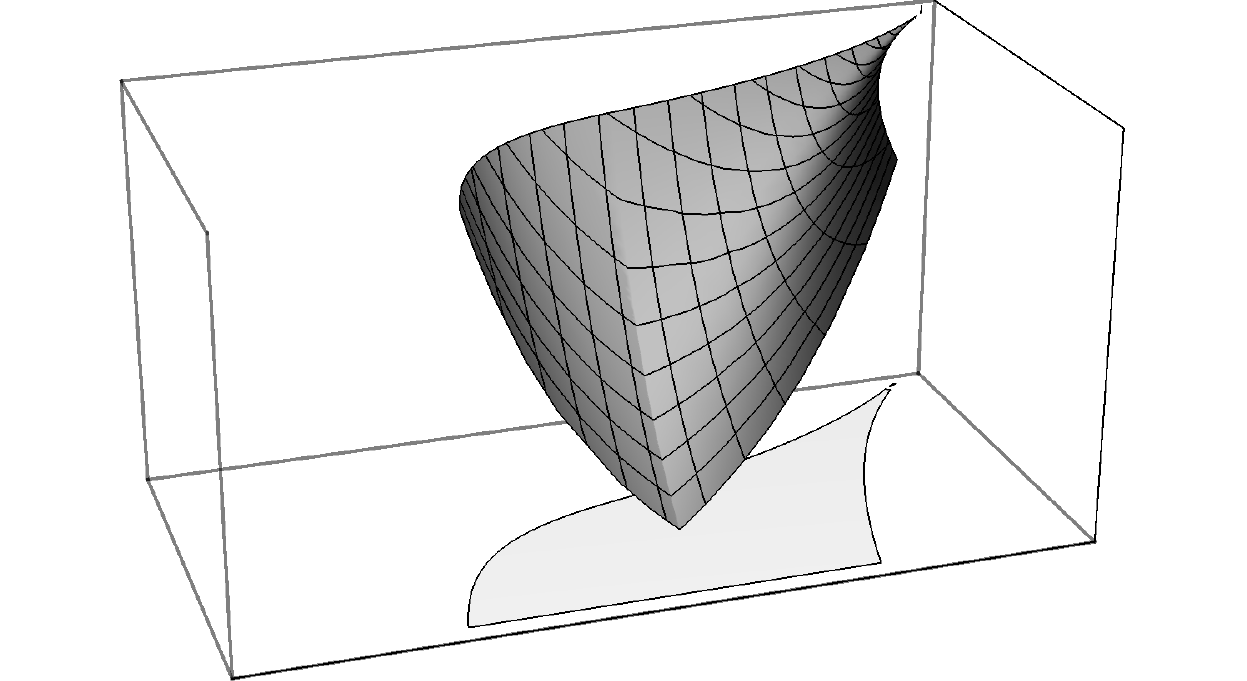}
  \qquad \quad
  \includegraphics[width=0.35\textwidth]{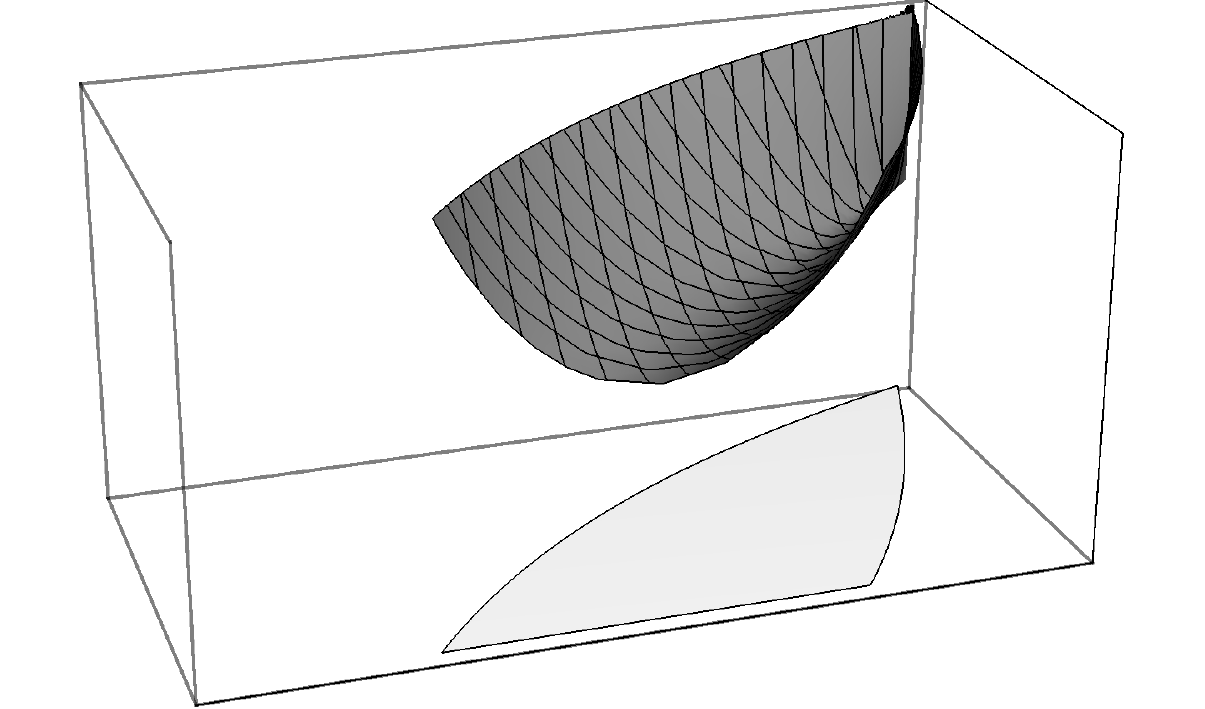}
  \qquad 
  \includegraphics[width=0.3\textwidth]{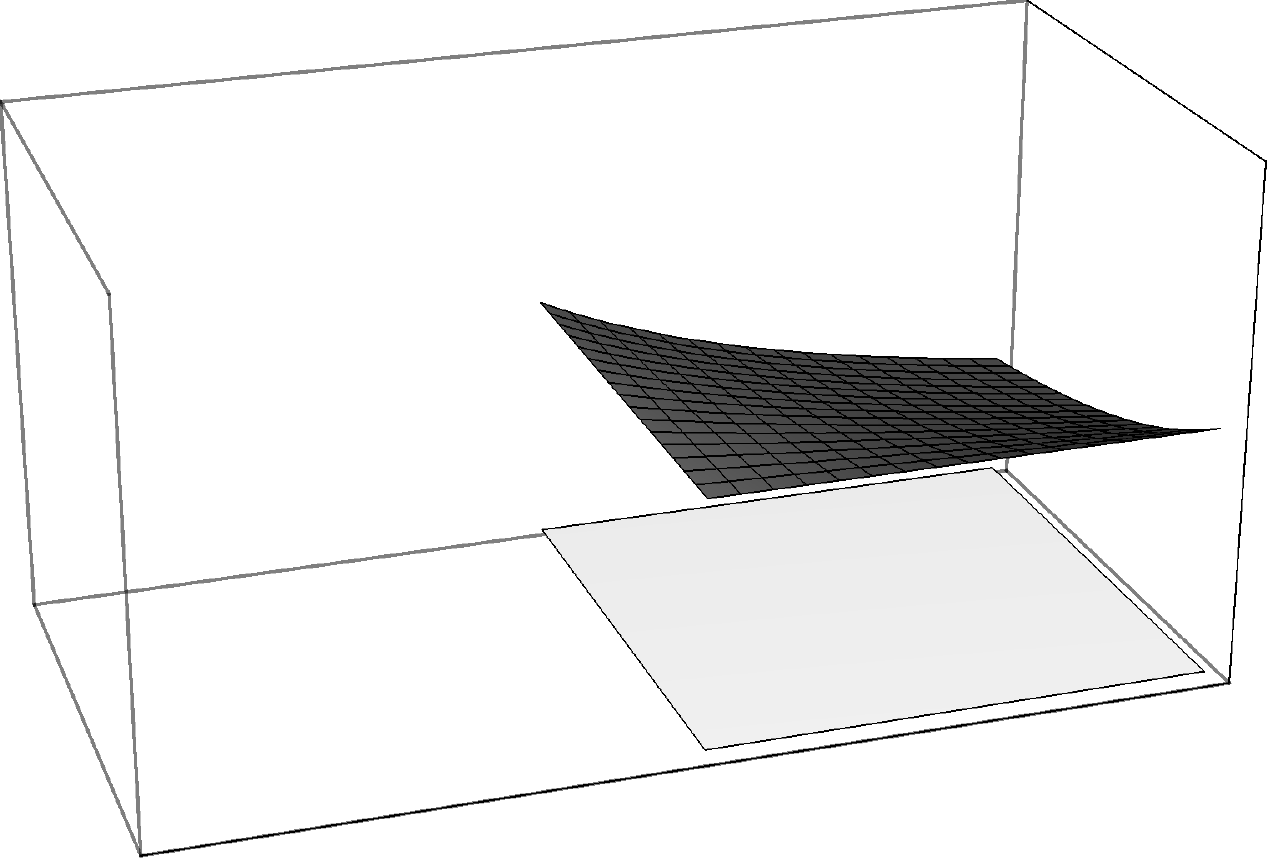}
  \caption{The functions $D_p(u,v)$ over $\Omega_p$ for $p=1,2,\infty$.}
\end{figure}

\begin{lemma} \label{lem:D-max}
  The function $D(u,v)$ is continuous on $\overline{\Omega}\setminus\{(1,1)\}$ and assumes its maximum value $1/\Delta$ on that set.
\end{lemma}

\begin{proof}
The continuity statement is clear from the definition of $D(u,v)$.
Say that a point $(u,v)$ in $\overline{\Omega}\setminus\{(1,1)\}$ is a critical point if $D(u,v)=1/\Delta$.
For $(u,v)\in \overline{\Omega}\setminus\{(1,1)\}$, let the points $P=(x,y)$ and $P'=(x',y')$ be such that $\Phi^{-1}(u,v) = \pmatrix xy{x'}{y'}$.
Then
\begin{equation*}
  D(u,v) = F\left(\Phi^{-1}(u,v)\right)^2 = F(P)^2  = (\det L)^{-1},
\end{equation*}
where $L$ is the lattice generated by the unit vectors $\frac{1}{F(P)}P$ and $\frac{1}{F(P')}P'$.
Thus $(u,v)$ is a critical point if and only if $L$ is a critical lattice, so by Lemma~\ref{lem:bound} all critical points lie on the boundary of $\Omega$.
If critical points in $\overline{\Omega}\setminus\{(1,1)\}$ exist, then we are done.

Suppose that a critical lattice $L$ corresponds to the point $(1,1)$ in the $uv$-plane.
Then there exists a $t$ such that
\begin{equation*}
  F_t(P) = F_t(P') = F_t(P+P'), \quad \text{where } P = \tfrac 1{\sqrt 2}(1,-1), P' = \tfrac 1{\sqrt 2}(1,1). 
\end{equation*}
Then the matrix $g=\frac 1{\sqrt 2}\pmatrix{2t^{-1}}{0}{t^{-1}}{-t}$ is in $\overline{\mathcal R}$ and satisfies $\Phi(g)=(0,\frac 12)$.
So if $(1,1)$ corresponds to a critical lattice $L$, then the point $(0,\frac 12)\in \overline\Omega$ is a critical point.
\end{proof}

That $\Omega = \Omega_{\mathcal S}$ follows from the next lemma.

\begin{lemma}
We have
\begin{align*}
	\mathcal S' &=  T \mathcal S, \\
	\mathcal S'' &= \left((-1,0]\times[0,1] \right) \setminus  M \mathcal S'.
\end{align*}
\end{lemma}

\begin{proof}
We begin by showing that $\mathcal S' =  T \mathcal S$.
For $(u,v)\in \mathcal S$ we have $\frac 12\leq u<1$, so \eqref{eqT} simplifies to $ T (u,v) = \left( \frac{1-u}{u}, \frac{1}{v+1} \right)$.
We show that the boundary of $\mathcal S$ maps to the boundary of $\mathcal S'$ under $ T $; the lemma then follows by the continuity of $ T $ on $\mathcal S$.
Since $ T ([\frac 12,1]\times \{0\}) = [0,1]\times \{1\}$ and $ T (\{1\}\times [0,1]) = \{0\}\times [\frac 12,1]$, it suffices to show that $ T (\partial) = \partial'$.
Suppose that $(u,v)\in \partial$ and that $\Phi(g) = (u,v)$.
Then
\begin{equation*}
	 T \circ \Phi\left(\pmatrix xy{x'}{y'}\right) = \left( \tfrac{1-u}{u}, \tfrac{1}{v+1} \right) = \left( \tfrac{-(y+y')}{y}, \tfrac{x}{x'+x} \right) = \Phi\left(\pmatrix{x+x'}{y+y'}{x}{y}\right).
\end{equation*}
Since this is clearly invertible, we conclude that $ T (\partial) = \partial'$.

We prove $\mathcal S'' = \left((-1,0]\times[0,1] \right) \setminus  M \mathcal S'$ similarly.
We have $ M ([0,1]\times \{1\}) = [-\frac 12,0]\times \{0\}$ and $ M (\{0\}\times [\frac 12,1]) = \{0\}\times [0,\frac 12]$ for the straight line segments, so it suffices to show that $ M (\partial') = \partial''$.
We will show that $( M \circ T )\partial = \partial''$, using that
\begin{equation*}
	( M \circ  T )(u,v) = \left(u-1,\tfrac{v}{v+1}\right).
\end{equation*}
Suppose that $(u,v)\in \partial$ and that $\Phi(g) = (u,v)$.
Then
\begin{equation*}
	 M \circ T \circ \Phi\left(\pmatrix xy{x'}{y'}\right) = \left(u-1,\tfrac{v}{v+1}\right) = \left( -\tfrac{(y+y')}{y'}, \tfrac{x'}{x+x'} \right) = \Phi \left(\pmatrix{x+x'}{y+y'}{x'}{y'}\right),
\end{equation*}
which completes the proof.
\end{proof}
Let $
\omega_\Sc=(1-\omega(\Sc))^{-1}\omega,
$
where $\omega$ was defined in (\ref{ome}).
\begin{lemma}\label{surf1}
Define $\mathcal S$ and $\Omega$ by \eqref{eq:S-def} and \eqref{Om}.
Then for almost all irrational $\a$ the sequence  $(\mu_m,\nu_m)$ is uniformly distributed over $\Omega$ with respect to the measure 
$\omega_\Sc$.
\end{lemma}
\begin{proof}
Since $\Omega=\Omega_\Sc$, Lemma~\ref{surf1}   follows  from Theorem~5.4.23 of \cite{DK} (see also \cite{Kra1}) and Theorem~\ref{sexp}.
\end{proof}

\subsection*{Proof of Theorem~\ref{t4}}
By Lemma~\ref{lem:D-max}, the function $\Delta D(u,v)$ assumes the value $1$ at some point in $\overline \Omega\setminus \{(1,1)\}$.
By Lemma~\ref{tl6} it follows that $\delta_F(\alpha)=1$ if and only if the sequence $(\mu_n,\nu_n)$ is infinitely often arbitrarily close to such a critical point.
It follows from Lemma~\ref{surf1} that $\delta_F(\alpha)=1$ for almost all $\alpha$.

To finish the proof it suffices to show that there are uncountably many such $\alpha$.
Since we already know Theorem~\ref{t4} is true in the case of the sup-norm, suppose that $F$ is not the sup-norm.
Then the lattice generated by $(1,0)$ and $(0,1)$ is not potentially critical, so we have $\Delta D(0,0)<1$.
Thus any $\alpha$ for which $(\mu_n,\nu_n)$ converges to $(0,0)$ has $\delta_F(\alpha)<1$, and there are uncountably many such $\alpha$ (for example, the set of $\alpha$ with strictly increasing partial quotients).\qed

\section{Concluding remarks}\label{relate}
 In addition to the proof of Theorem \ref{t4}, there are other  applications of the metric theory of $\Sc$-expansions and ergodic theory 
to quantities related to $\d_p(\a). $ For instance we may treat the distribution of the values of  \[\d_p(\a;m)\defeq \D_p\,D_p(\mu_m,\nu_m)\] from Theorem \ref{t6}.
For almost all $\a$  the distribution function 
\[
\lim_{M\rightarrow \infty}\tfrac{1}{M} \#\{1\leq m \leq M;\; \d_p(\a;m)\leq z\}
\] 
exists for all $z\in [0,1].$ For $p=1,2,\infty$ it can be evaluated explicitly, as was done for $p=\infty$ in Theorem 4 of  \cite{BJW}.
In particular,
for almost all $\a$ 
\[
\lim_{M\rightarrow \infty} \frac{1}{M}\sum_{1\leq m \leq M}\d_p(\a;m)=c_p
\] 
exists where
\[
c_1=\half(3-\log{4})=0.806853\dots,\;\;\;\;c_2=\tfrac{1}{ \log {3}}=0.910239\dots,\;\;\;\; c_\infty =\tfrac{1+\log 4}{\log 16}=0.860674\dots.
\]

It is well known that a close connection exists between dynamical systems associated to various kinds of continued fractions and the geodesic flow on $\SL(2,\Z)\backslash \SL(2,\R)$.  See \cite{EW} and a discussion in \cite{AS}  for more on this connection
and for references to the literature.
Roughly speaking,  the natural extension of a continued fraction transformation can be identified with a cross section for the geodesic flow. For example, the transformation $T$ from (\ref{eqT}) of the  regular continued fraction's natural extension gives a planar representation of the first return map and $\omega$ corresponds to the Liouville measure.
Geodesics can be identified with (proper classes of) indefinite  binary quadratic forms and  a cross section with a reduction domain. 
The trajectories we study in this paper correspond to cuspidal geodesics or, equivalently, forms with one rational root. 

Of course there is great interest in similar Diophantine problems about general indefinite forms
and hence general geodesic trajectories. 
A prime example is the Markov problem \cite{Mark}  about the minima of such forms and their possible values; these values determine  the Markov spectrum (see \cite{Bo}  and its  references).
 The  Lagrange spectrum is similarly defined using cuspidal trajectories;
it  is determined by the values of 
\[
\l(\a)=\liminf_{t\geq 1}\rho_t(\a),\;\;\;\text{where}\;\;\rho_t(\a)=t\min_{\substack{p\in \Z\\1\leq q\leq t}} |p- \a q|\;\;\text{for}\;\;t\geq 1.
\]

The Dirichlet  spectrum is determined  by  the  values of
$
\d(\a)=\limsup_{t\geq 1} \rho_t(\a);$   in \cite{Iv} it is defined to be the set of values of $\frac{\d(\a)}{1-\d(\a)}.$
There is a spectrum that is related to the Dirichlet spectrum in the same way that the Markov spectrum is related to the Lagrange spectrum. Like the Markov problem,  its study  involves general geodesic trajectories and their associated continued fractions. Again speaking roughly, we replace $\limsup$  over cuspidal geodesics  in the definition of $
\d(\a)$ by the supremum over all geodesics.
 Mordell \cite{Mor} introduced this problem (actually an $n$-dimensional version), which he posed as a kind of converse to Minkowski's linear forms theorem.   The case of two dimensions was treated in more detail  by Szekeres \cite{Sz}, Oppenheim \cite{Op1}  and Burger \cite{Bu}.
  This problem in higher dimensions 
 has also attracted a lot of attention (see e.g. \cite{Ra1,Ra2,SW}).
 
It should be apparent that a  general spectrum of this type can be defined for any strongly symmetric norm $F$,
not just the sup-norm, 
and that an associated reduction theory 
 for  indefinite binary quadratic forms can be developed that uses $F$-continued fractions.
For the 2-norm the problem was  introduced by Oppenheim \cite{Op2} and the relevant reduction theory was already found by Hermite.
Minkowski developed the reduction theory for the 1-norm with Hermite's theory in mind and certainly knew that a version could be based on the $p$-norm for a general $p$   \cite[footnote on p.~166]{Mink2}.
However, outside of the sup-norm, only isolated aspects of the spectrum and  reduction theory have been considered and only for the $p$-norm for  $p=1,2.$

\appendix
\section{Lemmas about norms}\label{lemmas}

Here we state and prove a number of simple technical lemmas that are referred to in the body of the paper.
Here $F$ is a norm on $\R^2$ with unit ball $\B$
and $P=(x,y),P'=(x',y') \in \R^2.$ For $t>0$ we define as above $F_t(x,y)=F(t^{-1}x,ty).$
%
%
The  lemmas give various properties of  norms
that satisfy the first condition (\ref{SS}) of strong symmetry.
Note that if $F$ satisfies (\ref{SS}) then so does $F_t$ for any $t>0.$
The first result is crucial and is used repeatedly in this paper.
\begin{lemma}\label{l2}
Suppose that $F$ satisfies (\ref{SS}).
If $|x'|\leq |x|$ and $|y'|\leq |y|$ then  we have that \[F(P')\leq F(P).\]
  \end{lemma}
\begin{proof}  To see this observe  that if $F(P)= s$ then $F(\pm x,\pm y)= s$ 
hence $F(x',y')\leq s$ by convexity.
\end{proof}
\begin{lemma}\label{lcom}
If $F,G,H$ satisfy (\ref{SS}) then so does  $ K$ defined by
\[
K(P)=H(F(P),G(P)).
\]
  \end{lemma}
\begin{proof}  
This follows easily using Lemma \ref{l2}.
\end{proof}

 \begin{lemma}\label{lll}
 Suppose that  $F$ satisfies (\ref{SS}). The following properties hold.
\begin{enumerate}[label=(\roman*)]
\item
If $F(P')\geq F(P)$ and $|y'|<|y|$
 then for some unique $t \geq 1$ we have 
 \[
 F_t(P')= F_t(P).
 \]
 \item
 If $F(P')\geq F(P)$ and $|x'|<|x|$
 then for some unique $t \leq 1$ we have 
 \[
 F_t(P')= F_t(P).
 \]
\end{enumerate}
\end{lemma}
\begin{proof}
We only prove (i) as (ii) is a consequence of (i) applied to the norm $G(x,y)=F(y,x).$

\noindent
{\it Existence:}
If $F(P')= F(P)$ take $t=1.$ Otherwise
for any $P\in \R^2$ define the continuous  function $f_P:[1,\infty)\rightarrow \R^+$ by
$f_P(t)=t^{-1}F_t(P).$
Now by Lemma \ref{l2}
\[
f_P(t)=F(t^{-2}x,y)\geq F(0,y)=|y|F(0,1).
\]
On the other hand, 
$f_{P'}(t)=F(t^{-2}x',y')\rightarrow F(0,y')=|y'|F(0,1)<|y|F(0,1)$ as $t \rightarrow \infty$.
Because $f_P(1)<f_{P'}(1) $
 the existence of desired $t$ follows by the intermediate value theorem.
 
\smallskip
\noindent{\it Uniqueness:}
Suppose that for $t_1\neq t_2$ with $t_1,t_2\geq 1$ we have
\[
F_{t_1}(P)=F_{t_1}(P')=F_{t_2}(P)=F_{t_2}(P').
\]
This implies that $|x|=|x'|$ and that $|y|=|y'|$, which is not true.
    \end{proof}
    
 The following result is trivial in case the norm is strictly convex.    
\begin{lemma}\label{equal}
Suppose that that $F$ satisfies (\ref{SS}),
that  we have  $F(P)=F(P')$ and that  $0<x'<x$ and $0<|y|<y'$. 
Then for any $d\geq1$
\begin{equation}\label{inin}
 F(P- d P')< F(P)+dF(P').
 \end{equation}
 \end{lemma}
 
 \begin{proof}
 To see this note first that in order for equality to hold in (\ref{inin}) we must have that 
 \begin{align*}
 F(P)+dF(P')= FP- P'-(d-1)P')\leq F(P-P')+(d-1)F(P'),
 \end{align*}
 which implies that
 \[
 F(P-P')\geq F(P)+F(P') \;\;\text{so that}\;\;\; F(P-P')= F(P)+F(P')
 \]
 hence 
 \[
  F(\half(P-P'))=\half( F(P)+F(P'))= F(P)=F(-P').
 \]
That this is impossible  follows by a simple convexity argument using the locations of \[P=(x,y)\;\;\;\text{and}\;\;\; -P'=(-x',-y'),\]
 together 
 with (\ref{SS}). \end{proof}

\begin{lemma}\label{ll2}
Suppose that that $F$ satisfies (\ref{SS}).
For  $\sigma,\sigma'\in [0,1]$ with $\sigma+\sigma'=1$ and $1\leq t_1\leq t_2$ we have
\[
F_{\sigma t_1+\sigma't_2}(P)\leq \sigma F_{t_1}(P)+\sigma'F_{t_2}(P).
\]
\end{lemma}
\begin{proof}
 Using the fact that the function $t\mapsto t^{-1}$ is concave up and applying  Lemma \ref{l2} we get that
 \begin{align*}
 F_{\sigma t_1+\sigma' t_2} (x,y)
\leq F\big(x(\tfrac{\sigma}{t_1}+\tfrac{\sigma'}{t_2}),y(\sigma t_1+\sigma't_2)\big).
 \end{align*}
 By the defining properties of a norm we finish the proof.
   \end{proof}


\begin{thebibliography}{1}
  
  \bibitem{AD} Andersen, N.  \& Duke, W.,  Markov spectra for modular billiards, to appear in {\it Math. Annalen} (2019).
  
  \bibitem{AS}  Arnoux, P. \& Schmidt, T.A., Cross sections for geodesic flows and $\a$-continued fractions. {\it Nonlinearity} 26 (2013), no. 3, 711–726.
  

\bibitem{Bo}  Bombieri, E., Continued fractions and the Markoff tree. {\it Expo. Math.} 25 (2007), no. 3, 187--213.

\bibitem{Bos}  	Bosma, W.,  Optimal continued fractions. { Nederl. Akad. Wetensch. Indag. Math.} 49 (1987), no. 4, 353--379.


\bibitem{BJW}  Bosma, W. \& Jager, H. \& Wiedijk, F.,  Some metrical observations on the approximation by continued fractions. {\it Nederl. Akad. Wetensch. Indag. Math.} 45 (1983), no. 3, 281–299. 






\bibitem{Bu}	Burger, E. B., On a question of Mordell and a spectrum of linear forms. {\it J. London Math. Soc.} (2) 62 (2000), no. 3, 701--715. 


\bibitem{Cas}	 Cassels, J. W. S., An introduction to the geometry of numbers. Die Grundlehren der mathematischen Wissenschaften in Einzeldarstellungen mit besonderer Berücksichtigung der Anwendungsgebiete, Bd. 99 Springer-Verlag, Berlin-Göttingen-Heidelberg 1959 viii+344 pp.

\bibitem{Co}  Cohn, H. Minkowski's conjecture on critical lattices in the metric $(|\xi|^p+|\eta|^p)^{\frac{1}{p}}$, {\it Ann. of Math.} 51 (1950) 734–738. 



\bibitem{DK} Dajani, K.  \& Kraaikamp, C., Ergodic theory of numbers. Carus Mathematical Monographs, 29. Mathematical Association of America, Washington, DC, 2002. x+190 pp.

\bibitem{DS}  Davenport, H. \& Schmidt, W. M.,  Dirichlet's theorem on diophantine approximation. 1970 Symposia Mathematica, Vol. IV (INDAM, Rome, 1968/69) pp. 113–132 Academic Press, London. 

\bibitem{DS2} 	Davenport, H. \& Schmidt, W. M., Dirichlet's theorem on Diophantine approximation. II. {\it Acta Arith.} 16 1969/1970 413--424.

\bibitem{Dav}  Davis, C. S., Note on a conjecture by Minkowski. {\it J. London Math. Soc.} 23, (1948). 172–175.
 
\bibitem{Dir}  Dirichlet, L.G.P.,  Verallgemeinerung eines Satzes aus der Lehre von den Kettenbr\"uchen nebst einigen Anwendungen auf die Theorie der Zahlen,  1842, Werke I, 633--638.



\bibitem{EW} 	Einsiedler, M. \& Ward, T., Ergodic theory with a view towards number theory. Graduate Texts in Mathematics, 259. Springer-Verlag London, Ltd., London, 2011. xviii+481 pp. 

\bibitem{Eul} Euler, L.,  De fractionibus continuis dissertatio,  Opera Omnia: Series 1, Volume 14, pp. 187 - 216 (1744)


\bibitem{GGM} 	Glazunov, N. M.\& Golovanov, A. S.\&  Malyshev, A. V.,  Proof of the Minkowski conjecture on the critical determinant of the region $|x|^p+|y|^p<1$. (Russian) Zap. Nauchn. Sem. Leningrad. Otdel. Mat. Inst. Steklov. (LOMI) 151 (1986), Issled. Teor. Chisel. 9, 40--53, 195; translation in J. Soviet Math. 43 (1988), no. 5, 2645–2653. 

\bibitem{GL}	Gruber, P. M. \& Lekkerkerker, C. G.,  Geometry of numbers. Second edition. North-Holland Mathematical Library, 37. North-Holland Publishing Co., Amsterdam, 1987. xvi+732 pp. 


\bibitem{Her0} Hermite C.,  Extraits de lettres de Mr. Ch. Hermite \`a M. Jacobi sur diff\'erents objets de la th\'eorie des nombres.  J. Reine Angew. Math.  40, 261--278 in Oeuvres I. 

 \bibitem{Her} Hermite,  C.  Sur l'introduction des variables continues dans la theorie des nombres. {\it J reine Angew Math. } 41:  (1851) 191–216.

\bibitem{Hum1}  Humbert,  G., Sur la m\'ethode d'approximation d'Hermite. {\it  J Math Pures Appl.} (7th Ser) 2: (1916) 70–103. 
\bibitem{Hum2} Humbert,  G.,  Sur les fractiones continues ordinaires et les formes quadatique binaires indéfinies. {\it J Math Pures Appl.} (7th Ser) 2: (1916) 104–154.


\bibitem{Iv} Ivanov, V. A.,  A theorem of Dirichlet in the theory of Diophantine approximations. (Russian) {\it Mat. Zametki}  24 (1978), no. 4, 459–474, 589.
English translation: Math. Notes 24 (1978), no. 3–4, 747–755 (1979).

\bibitem{Jag} Jager, H., Continued fractions and ergodic theory, transcendental numbers and related topics, RIMS Kokyuroko {\bf 599} (1986) no.1. 55--59.

\bibitem{Khi}  Khintchine, A.Ya., Continued fractions, English transl. by P. Wynn, Noordhoff, Groningen, (1963).



\bibitem{Lag}   Lagrange, J.L.,  Additions aux \'elements d'algebra d'Euler, Oeuvres VII.
.

\bibitem{IK} Iosifescu, M. \& Kraaikamp, C.,  Metrical theory of continued fractions. Mathematics and its Applications, 547 Kluwer Academic Publishers, Dordrecht, (2002) xx+383 pp.

\bibitem{Kra1} 	Kraaikamp, C.,  Statistic and ergodic properties of Minkowski's diagonal continued fraction. {\it Theoret. Comput. Sci.} 65 (1989), no. 2, 197--212. 
\bibitem{Kra2} 	Kraaikamp, C.,  A new class of continued fraction expansions. {\it Acta Arith.} 57 (1991), no. 1, 1--39.


\bibitem{Mah0} Mahler, K.  Lattice points in two-dimensional star domains. I. {\it Proc. London Math. Soc.} (2) 49, (1946) 128–157. 


\bibitem{Mah1} 	Mahler, K. On the minimum determinant and the circumscribed hexagons of a convex domain. {\it Nederl. Akad. Wetensch.,} Proc. 50, (1947) 692–703={\it Indagationes Math.} 9,  (1947) 326–337.





\bibitem{Mal} Malyšev, A. V., The application of an electronic computer to the proof of a certain conjecture of Minkowski from the geometry of numbers. (Russian) Modules and representations. Zap. Naučn. Sem. Leningrad. Otdel. Mat. Inst. Steklov. (LOMI) 71 (1977), 163–180, 286.

\bibitem{Mark} A. Markoff, A.,  Sur les formes quadratiques binaires ind\'efinies, Math. Ann. 15 (1879) 381--409, 17 (1880) 379--399. 
\bibitem{Mink1} Minkowski, H.,  Zur Theorie der Kettenbr\"uche,  {\it Ann. de l'\"Ecole Normale sup.}, ser 3. XIII, 41--60 (1896), in Gesammelte Abhandlungen I . 278--292.

\bibitem{Mink1.5} Minkowski, H.,  \"Uber die Ann\"aherung an eine reele Gr\"o\ss e durch rational Zahlen, Math. Annalen {\bf 54} 91--124 (1901), in Gesammelte Abhandlungen I, 320--352.

\bibitem{Mink1.7} Minkowski, H., Dichteste   gitterförmige   Lagerung   kongruenter   Körper,   Nachr.  K.  Ges. Wiss. Göttingen,  (1904) 311-355, in Gesammelte Abhandlungen.  Vol. II, Teubner,  Berlin,  (1911)  pp. 3--42.


\bibitem{Mink2} Minkowski, H.,  Geometrie der Zahlen, Teubner (1910)



\bibitem{Mink3} Minkowski, H.,  Diophantische Approximationen, 2d ed., Teubner, Leipzig, 1927, pp. 51--58.




\bibitem{Mor} Mordell, L. J.,  Note on an arithmetical problem on linear forms. {\it London Mathematical Society} 12 (1937): 34--6. 
\bibitem{Mor2} 	Mordell, L. J.,  Lattice points in the region $|A x^4+B y^4|\leq1$. {\it J. London Math. Soc.} 16, (1941) 152--156.


 \bibitem{Mosh} Moshchevitin, N.,  On Minkowski diagonal continued fraction. Analytic and probabilistic methods in number theory, 197–206, TEV, Vilnius, (2012).

\bibitem{Op1} Oppenheim, A.,  The continued fractions associated with chains of quadratic forms. {\it Proc. London Math. Soc.} (2) 44 (1938), no. 5, 323--335. 
%
\bibitem{Op2} 	Oppenheim, A., Two lattice-point problems. {\it Quart. J. Math.}, Oxford Ser. 18, (1947) 17--24. 
\bibitem{Per} 	Perron, O.,  Die Lehre von den Kettenbr\"uchen. Bd I. Elementare Kettenbr\"uche. (German) 3te Aufl. B. G. Teubner 
Verlagsgesellschaft, Stuttgart, (1954) vi+194 pp.


\bibitem{Ra1} Ramharter, G., \"Uber ein Problem von Mordell in der Geometrie der Zahlen. {\it Monatshefte fr Mathematik} 92 (1981): 143--60. 
\bibitem{Ra2} Ramharter, G.,  On Mordell's inverse problem in dimension three. {\it J. Number Th.}  58, (1996) 388--415.
%
\bibitem{Rei} Reinhardt, K.,  Über die dichteste gitterförmige lagerung kongruenter bereiche in der ebene und eine besondere art konvexer kurven. (German) {\it Abh. Math. Sem. Univ. Hamburg} 10 (1934), no. 1, 216–230. 

\bibitem{Ro}  Roy, D., On Schmidt and Summerer parametric geometry of numbers. {\it Ann. of Math.} (2) 182 (2015), no. 2, 739–786. 


\bibitem{Sch0} 	Schmidt, W. M.,  Diophantine approximation and certain sequences of lattices. {\it Acta Arith.} 18 1971 195--178.
\bibitem{Sch1}   Schmidt, W. M.,  Diophantine approximation. {\it Lecture Notes in Mathematics}, 785. Springer, Berlin, 1980. x+299 pp. 


\bibitem{SS}   Schmidt, W. M. \& Summerer, L., Parametric geometry of numbers and applications. {\it Acta Arith.} 140 (2009), no. 1, 67–91.

\bibitem{SS1} Schmidt, W. M. \& Summerer, L.,  Diophantine approximation and parametric geometry of numbers. {\it Monatsh. Math.} 169 (2013), no. 1, 51–104. 

%
%
%

\bibitem{Sie} Siegel, C.L.,  Lectures on the geometry of numbers. Notes by B. Friedman. Rewritten by Komaravolu Chandrasekharan with the assistance of Rudolf Suter. With a preface by Chandrasekharan. Springer-Verlag, Berlin, 1989. x+160 pp.
\bibitem{SW} Shapira, U. \& Weiss, B.,  On the Mordell-Gruber spectrum. {\it Int. Math. Res. Not.} IMRN 2015, no. 14, 5518--5559. 

\bibitem{Sz}  Szekeres, G.,  On a problem of the lattice plane. J. London Math. Soc. 12,  (1936) 88--93.

\bibitem{Ti}   Tietze, H.,  \"Uber die raschesten Kettenbruchentwicklungen reeller Zahlen, {\it Monatsh.
Math. Phys.}, 24,  (1913) 209--241.

\bibitem{Wa1} Watson, G. L.,  Minkowski's conjectures on the critical lattices of the region $|x|^p+|y|^p\leq 1$. I. {\it J. London Math. Soc.} 28, (1953). 305–309. 

\bibitem{Wa2}	Watson, G. L.,  Minkowski's conjectures on the critical lattices of the region $|x|^p+|y|^p\leq 1$.  II. {\it J. London Math. Soc.} 28, (1953). 402–410. 




  \end{thebibliography}
\end{document}